\documentclass[11pt,a4paper]{article}
\usepackage{fancyhdr}
\usepackage{amsmath}
\usepackage{amsthm}
\usepackage{amssymb}
\usepackage{wasysym}
\usepackage{paralist}
\usepackage{makeidx}
\usepackage[all]{xy}
\usepackage{mathdots}
\usepackage{yhmath}
\usepackage{mathtools}

\usepackage{young}
\usepackage[vcentermath]{youngtab}

\usepackage{graphicx}
\usepackage{epstopdf}

\input xy

\newcommand{\nc}{\newcommand}
 
 \nc{\cl}{\centerline}
 \renewcommand{\l}{{\rm len}}
 \nc{\SL}{{\rm SL}}
 \nc{\hatQ}{{\hat Q}}
 \nc{\sgn}{{\rm sgn}}
 
 \nc{\seee}{\mathbb C}

 \nc{\hatlambda}{{\hat\lambda}}
 
 \nc{\daggerlambda}{{\lambda^\dagger}}

\nc\diag{{\rm diag}}
\renewcommand{\vert}{{\,|\,}}
\nc{\hatL}{{\hat L}}

\nc{\barE}{{\bar   E}}
\nc{\D}{{\mathcal D}}
\nc{\E}{{\mathcal E}}
\nc{\F}{{\mathcal F}}
\nc{\FF}{{\mathcal F}}
\nc{\I}{{\mathcal I}}
\nc{\even}{{\rm e}}
\nc{\ep}{\epsilon}
\nc{\odd}{{\rm o}}
\nc{\Coker}{{\rm Coker}}
\nc{\olE}{{\overline E}}
\nc{\indBG}{{\rm ind}_B^G\,}
\nc{\indHG}{{\rm ind}_H^G\,}

\nc{\que}{{\mathbb Q}}
\nc{\barlambda}{{\bar\lambda}}
\nc{\barmu}{{\bar\mu}}
\nc{\barnu}{{\bar\nu}}
\nc{\bartau}{{\bar\tau}}
\nc{\barm}{{\bar m}}
\nc{\divind}{{\rm div.ind}}
\nc{\tl}{{\tilde{\lambda}}}
\nc{\dar}{\downarrow}

\nc{\Sym}{{\rm \Sigma}}
\nc{\Symm}{{\rm Sym}}

\newcommand{\q}{\quad}

\newcommand{\nat}{{\mathbb N}}

\renewcommand{\mod}{{\rm mod}}
\newcommand{\iso}{\cong}

\newcommand{\Sp}{{\rm Sp}}

\newcommand{\bs}{\bigskip}

\renewcommand{\vert}{\,|\,}

\renewcommand{\sgn}{{\rm sgn}}

\newcommand{\reg}{{\rm reg}}

\xyoption{all}

\newcommand{\ind}{{\rm ind}}
\renewcommand{\vert}{\,|\,}
 
\newcommand{\zed}{{\mathbb Z}}
\newcommand{\Ext}{{\rm Ext}}
\newcommand{\End}{{\rm End}}
\newcommand{\Hom}{{\rm Hom}}
\newcommand{\cf}{{\rm cf}}
\newcommand{\soc}{{\rm soc}}
\newcommand{\hd}{{\rm hd}}
\renewcommand{\mod}{{\rm mod}}

\newcommand{\GL}{{\rm GL}}

\newcommand{\res}{{\rm{res\,}}}

\newcommand{\Mull}{{\rm Mull}}
\renewcommand{\mod}{{\rm{mod}}}

\nc{\geom}{{\rm geom}}
\nc{\rep}{{\rm rep}}
\newcommand{\rad}{{\rm rad}}
\newcommand{\core}{{\rm core}}
\newcommand{\cont}{{\rm cont}}
\newcommand{\Res}{{\rm Res}}

\renewcommand{\P}{{\mathcal  P}}

\newtheorem{definition}{Definition}[section]
\newtheorem{proposition}[definition]{Proposition}

\newtheorem{lemma}[definition]{Lemma}

\begin{document}


\centerline{\bf Invariants of Specht modules}

\bigskip

\centerline{Stephen Donkin  and Haralampos Geranios}

\bigskip

{\it Department of Mathematics, University of York, York YO10 5DD}\\

\medskip

{\tt stephen.donkin@york.ac.uk,  haralampos.geranios@york.ac.uk}

\bs

\centerline{10 April  2015}
\bs\bs\bs

\section*{Abstract}

\q In  \cite{H}   Hemmer  conjectures that the module  of fixed points for the symmetric group $\Sigma_m$ of a Specht module for $\Sym_n$ (with $n>m$), over a  field of positive characteristic $p$, has a Specht series, when viewed as a $\Sigma_{n-m}$-module. We provide a counterexample  for each prime $p$. The examples have the same form for $p\geq 5$ and we treat the cases $p=3$ and $p=2$ separately.

\section*{Introduction}

\q  Let $k$ be a field of positive characteristic. For a positive integer $n$  we denote by $\Sym_n$ the symmetric group of degree $n$.  Let  $m,n$ be positive integers with $m<n$. The space $X^{\Sym_m}$ of points fixed by $\Sym_m$ of a module $X$ over $k\Sym_n$ is naturally a module for   $k\Sym_{n-m}$.   It is conjectured in \cite{H}, Conjecture 7.2   (see also \cite{H2}, Conjecture 7.3) that this fixed point module has a Specht filtration. In case $m<p$ this follows by the exactness of the $\Sym_m$ fixed point functor. Here we study 
 the case $m=p$ and provide counterexamples of Hemmer's Conjecture.  The layout of the paper is the following. Section 1 is preliminary: we establish notation and remind the reader of connections between the representation theory of the general linear groups and symmetric groups.  We also outline the strategy for finding invariants which is pursued in what follows.  In Sections $2$ and $3$ we consider the case $p\geq 5$.  The main purpose of Section $2$ is to establish the dimensions of invariants for certain three  part partitions with first part at most $p$.   This is then used in Section $3$ to show that the Specht module corresponding to the partition $(p,p,p)$ is a counterexample. In Section 4 we carry out a similar analysis in the case $p=3$ and show that the Specht module corresponding to the partition $(4,4,4)$ is a counterexample. In Section 4 we carry out a similar analysis in the case $p=2$ and show that the Specht module corresponding to the partition $(4,4)$ is a counterexample.  In Sections $2$ and $3$ our arguments are of a general nature. In Sections $4$ and $5$, as well as general arguments, we found it convenient to use specific results on the nature of Specht modules to be found in (or easily deducible from) James's tables, \cite{James},  Appendix.

\bs\bs\bs\bs


\section{Preliminaries}

\q We recall  certain combinatorics associated to partitions,  the usual set-up for the representation theory of general linear groups and symmetric groups and also some well-known results.  We   then describe some results that will be particularly useful to us. We  conclude this section by describing  the   methodology to be followed  in the rest of the paper.

\subsection{Combinatorics}

\q By a partition we mean an infinite  sequence $\lambda=(\lambda_1,\lambda_2,\ldots)$ of nonnegative integers with $\lambda_1\geq\lambda_2\geq \ldots$ and $\lambda_j=0$ for $j$ sufficiently large.   If $m$ is a positive integer such that $\lambda_j=0$ for $j>m$ we identify $\lambda$ with the finite sequence $(\lambda_1,\ldots,\lambda_m)$.  The length   $\l(\lambda)$ of  a partition $\lambda=(\lambda_1,\lambda_2,\ldots)$ is $0$ if $\lambda=0$ and  is the positive integer $m$ such that  $\lambda_m\neq 0$, $\lambda_{m+1}=0$, if $\lambda\neq 0$. For a partition $\lambda$, we denote  by $\lambda'$ the transpose partition of $\lambda$. We write $\P$ for the set of partitions.  We define the   degree of a partition $\lambda=(\lambda_1,\lambda_2,\ldots)$ by $\deg(\lambda)=\lambda_1+\lambda_2+\cdots$.

\q We set $X(n)=\zed^n$. There is a natural partial order on $X(n)$. For $\lambda=(\lambda_1,\ldots,\lambda_n), \mu=(\mu_1,\ldots,\mu_n)\in X(n)$,  we write $\lambda\leq \mu$ if $\lambda_1+\cdots+\lambda_i\leq \mu_1+\cdots+\mu_i$ for $i=1,2,\ldots,n-1$ and $\lambda_1+\cdots+\lambda_n=\mu_1+\cdots+\mu_n$. We shall use the standard $\zed$-basis   $\ep_1,\ldots,\ep_n$ of   $X(n)$,  so  $\ep_i=(0,\ldots,1,\ldots,0)$ (with $1$ in the $i$th position). We write $\Lambda(n)$ for the set of $n$-tuples of nonnegative integers. 

\q We write $X^+(n)$ for the set of dominant $n$-tuples of integers, i.e., the set of elements $\lambda=(\lambda_1,\ldots,\lambda_n)\in X(n)$ such that $\lambda_1\geq \cdots\geq  \lambda_n$.  We write  $\Lambda^+(n)$ for the set of partitions into at most $n$-parts, i.e.,  $\Lambda^+(n)=X^+(n)\bigcap \Lambda(n)$. We shall sometimes refer to elements of $\Lambda(n)$ as polynomial weights and elements of $\Lambda^+(n)$ as polynomial dominant weights. For a nonnegative integer $r$ we write $\Lambda^+(n,r)$ for the set of partitions of $r$ into at most $n$ parts, i.e., the set of elements of $\Lambda^+(n)$ of degree $r$.

\q We fix a prime number $p$. A partition $\lambda=(\lambda_1,\lambda_2,\ldots)$ is $p$-regular if there is no positive integers $i$ such that $\lambda_i=\lambda_{i+1}=\cdots=\lambda_{i+p-1}>0$. We write $\P_\reg$ for the set of $p$-regular partitions  and $\P_\reg(r)$ for the set of $p$-regular partitions  of degree $r$.

 \q   We write $X_1(n)$ for the set of $p$-restricted partition into at most $n$ parts, i.e., the set of elements $\lambda=(\lambda_1,\ldots,\lambda_n)\in \Lambda^+(n)$ such that $0\leq \lambda_1-\lambda_2,\ldots,\lambda_{n-1}-\lambda_n, \lambda_n<p$. Note that an element $\lambda\in \Lambda^+(n)$ belongs to $X_1(n)$ is and only if the conjugate  $\lambda'$ is $p$-regular.

\subsection{Rational Modules and Polynomial Modules}

\q   Let $k$ be an algebraically closed field of positive characteristic $p$.  A good reference for  the polynomial representation theory of $\GL_n(k)$ is the monograph \cite{EGS}. Further details may also be found in the monograph, \cite{D3}.  (To apply this here one should take $q=1$ in the set-up considered there.)  

\q  If $V,W$ are vector spaces over $k$, we write $V\otimes W$ for the tensor product $V\otimes_k W$. For an affine algebraic group $G$ over $k$ and rational modules $V,W$ we write $\Hom_G(V,W)$ for the space of $G$-module homomorphisms and write $\Ext^i_G(V,W)$ for the extension spaces, $i\geq 0$.  We also denote by $k$ the trivial one dimensional $G$-module. We write $H^0(G,W)$ for the space of fixed points and $H^i(G,W)$ for the cohomology space $\Ext^i_G(k,W)$, $i\geq 0$. 

\q For a finite dimensional rational module $V$ and simple rational $G$-module $L$ we write $[V:L]$ for the multiplicity of $L$ as a composition factor of $V$.

\q If $V$ is a rational $G$-module and $H$ is a closed group of $G$ we write $\Res_H (V)$ for $V$ regarded as an $H$-module by restriction. 

\q We write $G$ for the general linear group $\GL_n(k)$, write $T$ for the maximal torus of $G$ consisting of diagonal matrices,  write  $B$ for the subgroup of $G$ consisting of all invertible lower triangular matrices and write $U$ for the  subgroup of $G$ consisting of all lower unitriangular matrices. Thus  $B$ is a Borel subgroup of $G$ with maximal unipotent subgroup $U$ and $B$ is the semidirect product of $U$ and $T$.  

\q We form the integral group ring $\zed X(n)$. This has $\zed$-basis of formal exponentials $e^\lambda$,  which multiply according to the rule $e^\lambda e^\mu=e^{\lambda+\mu}$, $\lambda,\mu\in X(n)$.   We identify $\lambda=(\lambda_1,\ldots,\lambda_n)\in X(n)$ with the multiplicative character of $T$ taking an element $t$ with diagonal entries $t_1,\ldots,t_n$ to $\lambda(t)=t_1^{\lambda_1}\ldots t_n^{\lambda_n}$ and thereby identify $X(n)$ with the character group of $T$.  We denote by $k_\lambda$ a one dimensional rational $T$-module on which $t\in T$ acts as multiplication by $\lambda(t)$. The modules  $k_\lambda$, $ \lambda\in X(n)$,  form  a complete set of pairwise non-isomorphic rational $T$-modules. For $\lambda\in X(n)$, the action of $T$ extends uniquely to a module action of $B$ on $k_\lambda$ and the modules $k_\lambda$, $\lambda\in X(n)$,  form  a complete set of pairwise non-isomorphic irreducible rational $B$-modules. 

\q For each $\lambda\in X^+(n)$ there is an irreducible rational $G$-module $L(\lambda)$ which has unique highest weight $\lambda$ and such $\lambda$ occurs as a weight with multiplicity one. The modules $L(\lambda)$, $\lambda\in X^+(n)$, form a complete set of pairwise non-isomorphic irreducible rational  $G$-modules. 

\q For $\lambda\in X^+(n)$ we write $\nabla(\lambda)$ for the induced module $\ind_B^Gk_\lambda$.  
Then $\nabla(\lambda)$ is finite dimensional  and its $\chi(\lambda)$ character is the Schur symmetric function corresponding to $\lambda$. The $G$-module socle of $\nabla(\lambda)$ is $L(\lambda)$. The module $\nabla(\lambda)$ has unique highest weight $\lambda$ and this weight occurs with multiplicity one.

\q For  $\lambda\in X^+(n)$ we write  $\Delta(\lambda)$ for the corresponding Weyl module, i.e.,  for the contravariant dual (in the sense of \cite{EGS}) of $\nabla(\lambda)$.

\q A filtration $0=V_0\leq V_0\leq \cdots\leq V_r=V$ for a finite dimensional rational $G$-module $V$ is said to be good if for each $1\leq i\leq r$ the quotient $V_i/V_{i-1}$ is either zero or isomorphic to $\nabla(\lambda^i)$ for some $\lambda^i\in X^+(n)$.  For a rational $G$-module $V$ admitting a good filtration for each $\lambda\in X^+(n)$, the multiplicity $|\{1\leq i\leq r\vert V_i/V_{i-1}\cong \nabla(\lambda)\}|$ is independent of the choice of the good filtration and will be denoted $(V:\nabla(\lambda))$. 

\q For $\lambda,\mu\in X^+(n)$ we have $\Ext^1_G(\nabla(\lambda),\nabla(\mu))=0$ unless $\lambda>\mu$,  see e.g. \cite{D4},  Lemma 3.2.1. It  follows that if $V$ has a good  filtration 
$0=V_0\leq V_0\leq \cdots\leq V_t =V$ with sections $V_i/V_{i-1}\cong \nabla(\lambda_i)$, $1\leq i\leq t$, and $\mu_1,\ldots,\mu_t$ is a reordering of the $\lambda_1,\ldots,\lambda_t$ such that $\mu_i<\mu_j$ implies that $i<j$ then there is a good filtration $0=V_0'<V_1'<\cdots <V_t'=V$ with $V_i'/V_{i-1}'\cong \nabla(\mu_i)$, for $1\leq i\leq t$.

\q We write $k[G]$ for the coordinate algebra of $G$.  For $1\leq i,j\leq m$ let $c_{ij}$ denote the corresponding coordinate function, i.e., the function taking $g\in G$ to its $(i,j)$ entry. We set $A(n)=k[c_{11},\ldots,c_{nn}]$.  Then $A(n)$ is a free polynomial algebra in generators $c_{11},\ldots,c_{nn}$.  Moreover  $A(n)$ has a $k$-space decomposition 
$A(n)=\bigoplus_{r=0}^\infty A(n,r)$, where  $A(n,r)$ is the span of the all monomials  $c_{i_1j_2}c_{i_2j_2}\ldots c_{i_rj_r}$.

\q   Let $V$ be a finite dimensional rational module with basis $v_1,\ldots,v_m$. The coefficient functions $f_{ij}$, $1\leq i,j\leq m$, are defined by the equations
 $$gv_i=\sum_{j=1}^m f_{ji}(g)v_j$$
 for $g\in G$, $1\leq i\leq m$.  The coefficient space $\cf(V)$ is the subspace of $k[G]$ spanned by the coefficient functions 
 $f_{ij}$, $1\leq i,j\leq m$. 
 
 \q A $G$-module $V$ is called polynomial if $\cf(V)\leq A(n)$ and polynomial of degree $r$ if $\cf(V)\leq A(n,r)$.  A polynomial $G$-module $V$ has a unique module decomposition 
$$V=\bigoplus_{r=0}^\infty V(r)$$
where $V(r)$ is polynomial of degree $r$. The coordinate algebra $k[G]$ has a natural Hopf algebra structure and each space $A(n,r)$ is a subcoalgebra. The dual space $S(n,r)$ has a natural algebra structure. The algebras $S(n,r)$ are called Schur algebras. The category of polynomial modules of degree $r$ is naturally equivalent to the category of $S(n,r)$-modules.

 \q The modules $L(\lambda)$, $\lambda\in \Lambda^+(n)$, form a complete set of pairwise non-isomorphic polynomial $G$-modules. For a finite dimensional rational $G$-module $V$ and $\lambda\in X^+(n)$ we write $[V:L(\lambda)]$ for the multiplicity of $L(\lambda)$ as a composition factor of $V$. 
For $\lambda\in \Lambda^+(n)$ we write $I(\lambda)$ for the injective hull of $L(\lambda)$ in the category of polynomial $G$-modules. Then $I(\lambda)$ is finite dimensional, indeed it is may be identified with the injective hull of $L(\lambda)$ as a module for the Schur algebra $S(n,r)$, where $r=|\lambda|$. Moreover, for $\lambda\in  \Lambda^+(n,r)$, the module $I(\lambda)$ has a good filtration and we have the reciprocity formula $(I(\lambda):\nabla(\mu))=[\nabla(\mu):L(\lambda)]$ see e.g., \cite{DStd}, Section 4, (6).

\q  It will be of great practical use  to know that $\Ext^1_{G}(\nabla(\lambda),\nabla(\mu))=0$ when $\lambda$ and $\mu$ belong to different blocks. Here  the relationship with cores of partitions diagrams (discussed later) will be crucial for us. For a partition $\lambda$ we denote by $[\lambda]$ the corresponding partition diagram (as in \cite{EGS}). The $p$-core of $[\lambda]$ is the diagram obtained by removing skew $p$-hooks,  as in \cite{James}. If $\lambda,\mu\in \Lambda^+(n,r)$ and $[\lambda]$ and $[\mu]$ have different $p$-cores then the simple modules $L(\lambda)$ and $L(\mu)$ belong to different blocks , see e.g.  \cite{D2}, Section 1, (7), and it follows in particular that $\Ext^i_{S(n,r)}(\nabla(\lambda),\nabla(\mu))=0$, for all $i\geq 0$.

\q We write $E$ for the natural $G$-module of column vectors. Then $E$ has basis $e_1,\ldots,e_n$, where $e_i$ is the column vector with $1$ in the $i$th position and zeros elsewhere, for $1\leq i\leq n$.  For a partition $\lambda=(\lambda_1,\lambda_2,\ldots)$ we write $S^\lambda E$ for the tensor product of symmetric powers $S^{\lambda_1}E\otimes S^{\lambda_2} E\otimes \cdots$.

\bs

\subsection{Representations of the Symmetric Group}

\q We write $\Sym_r$ form the symmetric group on $\{1,2,\ldots,r\}$, for $r\geq 1$.  A good reference for the representations of the symmetric groups is James's  monograph \cite{James}. 

\q We assume now that $n\geq r$. We write $\mod\,  S(n,r) $ for the category of finite dimensional left $S(n,r)$-modules and $\mod\, \Sym_r$  for the category of finite dimensional $k\Sym_r$-modules.

  By \cite{EGS}, Sections 6.1-6.4 we have the Schur functor  

$$f: \mod \ S(n,r)\rightarrow \mod\, \Sym_r$$

which relates the polynomial representations of $\GL_n(k)$ of degree $r$ with representations  of $\Sym_r$ over $k$  and is given on objects  by 
 $f\,V=V^{\omega_r}$.  For $\lambda$ a partition of degree $r$ we denote  by $\Sp(\lambda)$ the corresponding Specht module and by $M(\lambda)$ the corresponding permutation module for $\Sym_r$. By \cite{EGS}, Sections 6.3-6.4 and  \cite{D1}, 2.1, (20)(i) we have the following results.

\bigskip

{\bf Proposition 1.3.1} \sl The functor $f$ has the following properties :\\

(i)) for  $\lambda\in \Lambda^+(n,r)$,  we have $f S^\lambda E= M(\lambda)$ and   $f\, \nabla(\lambda)=\Sp(\lambda)$;\\
(ii) for  $\lambda\in \Lambda^+(n,r)$ we have  $f\, L(\lambda) \neq 0$ if and only if $\lambda\in X_1(n)$ and the set $\{f\, L(\lambda) \vert  \lambda\in X_1(n)\}$ is a full set of  pairwise non-isomorphic simple $\Sym_r$-modules.

\rm 
\bigskip

\q By \cite{CL} Theorem 3.7 and \cite{D1} Propositions 10.5 and 10.6 we have the following results.

\bigskip

 {\bf Lemma 1.3.2} \sl Let $p\geq 3$. For $\lambda,\mu\in \Lambda^+(n,r)$ we have,\\
(i) $\Ext^i_{S_k(n,r)}(\nabla(\lambda),\nabla(\mu))=\Ext^i_{\Sym_r}(\Sp(\lambda),\Sp(\mu))$ for $0\leq i<p-2$.\\
(ii) If $\Hom_{\Sym_r}(\Sp(\lambda),\Sp(\mu))\neq0$ then $\lambda\geq \mu$.

(iii) $\End_{\Sym_r}(\Sp(\lambda))=k$.

\rm

\bigskip

\q An alternative description of  the simple modules for  the symmetric group $\Sym_r$ can be found in  \cite{James}, Theorem 11.5.  In this description the simple modules for  $\Sym_r$ are parametrized by the set of $p$-regular partitions of $r$. More precisely, for $\lambda$ a  $p$-regular partition of $r$,  the module $\Sp(\lambda)$ has a unique simple quotient  $D^\lambda=\Sp(\lambda)/\rad(\Sp(\lambda))$ and the set $\{D^\lambda \vert \lambda\in \P_\reg(r)\}$ is a full set of pairwise  non-isomorphic simple $\Sym_r$-modules. Moreover for $\mu\in \P_\reg(r)$ with $[\Sp(\lambda):D^\mu]\neq0$ we have that $\mu\geq\lambda$.

\bs

\q These two different descriptions are related by the rule \cite{EGS}, Remark 6.4:

$$f\, L(\lambda) \cong  k_\sgn\otimes  D^{\lambda'}$$

where $k_\sgn$ denotes  the 1-dimensional  module afforded by  the  sign representation of $\Sym_r$.

\q Therefore a direct relation between  the two descriptions of the irreducible modules for the symmetric group is  described in terms of the involution   $\P_\reg(r)\to \P_\reg(r)$, $\lambda\mapsto \tilde\lambda$  defined by 
$k_\sgn\otimes D^\lambda\cong D^{\tilde\lambda}$. This bijection is named after G. Mullineux,  who proposed, in \cite{Mull},  an algorithm to describe it explicitly. We write $\Mull:\P_\reg(r)\to \P_\reg(r)$ for this bijection and call it the Mullineux involution. Thus we have
$$f \, L(\lambda) \iso D^{\Mull(\lambda')}$$
for $\lambda$ a $p$-restricted partition of degree $r$.

\q Mullineux's conjecture was proved by Ford  and Kleshchev in \cite{FK}. 

\bs

\q Recall that for $r\leq n$ we also have Green's inverse functor

 $$g:\mod (k\Sym_r)\to \mod(S(n,r))$$
 
which is defined in the following way. Set $S=S(n,r)$  and let $e\in S$ be the idempotent defined in \cite{EGS}. We identify   $k\Sym_r$ with $eSe$, as in \cite{EGS} . Then we have  the inverse Schur functor $g:\mod(eSe)\to \mod(S)$, $gX=Se\otimes_{eSe} X$. 

Recall that $fS^\lambda E= M(\lambda)$,  for $\lambda\in \Lambda^+(n,r)$, $r\leq  n$.  We  shall also need the reverse direction, namely that $g M(\lambda)=S^\lambda E$, for $\lambda\in \Lambda^+(n,r)$.

\bs

{\bf Lemma 1.3.3} \q\sl  Suppose $p\geq 5$.

(i) If $X$ is a finite dimensional polynomial module of degree $r\leq n$ and $X$ has a good filtration then $gfX=X$.

(ii) For $\lambda\in \Lambda^+(n,r)$, $r\leq n$, we have $g\, M(\alpha)=S^\alpha E$.

(iii) For $\lambda\in \Lambda^+(n,r)$, an $S(n,r)$-module $Y$ and $0\leq i<  p-2$ we have
$$\Ext^i_{S(n,r)}(S^\lambda E,Y)= \Ext^i_{\Sym_r}(M(\lambda),fY).$$

\rm
\bs

\begin{proof}   We note that,  for $\lambda\in \Lambda^+(n,r)$,  we have $g\, \Sp(\lambda)=\nabla(\lambda)$, \cite{D1}, Proposition 10.6 (i).    Now let $Y= fX$. Then by the universal property of the tensor product construction the inclusion map $Y\to X$ induces the map $\theta: gY\to X$ such that $f\circ \theta: fgY\to X$ is inclusion.  In particular the image  of $\theta$ contains $fX$ and $fX$ generates $X$, by \cite{D3}, Proposition 4.5.4.  Hence $\theta$ is surjective. We now choose a filtration $X=X_1>X_2>\cdots > X_{m+1}$ with $X_i/X_{i+1}\cong \nabla(\lambda(i))$, $1\leq i\leq m$, for partitions $\lambda(1),\ldots,\lambda(m)$ of $r$ into at most $n$ parts. 
By the right exactness of  $g$, we have 

\begin{align*}\dim gfX&\leq \sum_{i=1}^m  \dim gf(X_i/X_{i-1})=\sum_{i=1}^m \dim g\, \Sp(\lambda(i))\cr
&= \sum_{i=1}^m \dim \nabla(\lambda(i))= \dim X
\end{align*}
so that $\dim X\geq \dim g f X$ and hence the surjection  $\theta$ is an isomorphism.

(ii) Take $X=S^\lambda E$ in (i).

(iii) Take $X=M(\lambda)$ in \cite{D1}, Proposition 10.5(ii).

\end{proof}

{\bf Remarks 1.3.4}    (i) The results (and arguments) just given are valid also in the quantised context of \cite{D1}.

(ii)   The corresponding result for modules with a Weyl filtration is given in \cite{HN},  Theorem 3.8.1.

\bs

\q We collect together some properties of nodes that will be used in what follows.   Let $\lambda$ be a partition.  

(i)  We call a node $R$ of $\lambda$ (or more precisely of the diagram of $\lambda$)    {\it removable} if the removal of $R$ from the diagram of $\lambda$ leaves the diagram of a partition, which will be denoted $\lambda_R$. Thus the node  $R$ is removable node if it has the form $(i,\lambda_i)$ for some $1\leq i\leq  \l(\lambda)$ and either $i=\l(\lambda)$ or 
$\lambda_i >\lambda_{i+1}$. 

(ii) An addable node $A$ of $\lambda$ is an element of $\nat\times \nat$ such that the addition of $A$ to the diagram of $\lambda$ gives the diagram of a partition, which will be denoted $\lambda^A$. Thus $A$ is addable if it has the form $(i,\lambda_i+1)$ for some $1\leq i\leq \l(\lambda)$ and either $i=1$ or $\lambda_i<\lambda_{i-1}$ or  if it is  $A=(\l(\lambda)+1,1)$. 

(iii)  The {\it residue} of a node $A=(i,j)$ of a partition $\lambda$ is defined to be the congruence class of  $j-i$ modulo $p$. 

(iv) Let $A$ and $B$ be removable or addable nodes of of $\lambda$. We  shall say that  $A$ is lower than $B$ if  $A=(i,r)$, $B=(j,s)$ and $i>j$. 

(v) Let $R$ be a removable node of $\lambda$ with $\res(R)=\alpha$. We say that $R$ is {\it normal} if for every addable node $A$ above $R$ with $\res(A)=\alpha$ there exists a removable node $C(A)$ strictly between $A$ and $R$ with residue $\alpha$ and $C(A)\neq C(A')$ for $A\neq A'$.

(vi) The {\it residue content} of the Young diagram of $\lambda$, denoted  $\cont(\lambda)$,  is defined to be the $p$-tuple,
$(c_0,c_1,\dots,c_{p-1})$ where $c_\alpha=\cont_\alpha(\lambda)$ is the number of nodes in $[\lambda]$ which have residue 
$\alpha$ for $\alpha=0,1,\dots,p-1$.

\bs

\q The following  is   known as  Nakayama's Conjecture \cite{JK}  6.1.21.  Let $\lambda,\mu\in \P_\reg(r)$. The irreducible $\Sym_r$-modules  $D^\lambda$ and $D^\mu$ belong to the same block if and only if $[\lambda]$ and $[\mu]$ have the same $p$-core. Moreover, any Specht module belongs to a  block of $\Sym_r$ and for $\lambda,\mu\in \P(r)$  the corresponding Specht modules $\Sp(\lambda)$ and $\Sp(\mu)$ belong to the same block if and only if $[\lambda]$ and $[\mu]$ have the  same $p$-core.  An equivalent description of the blocks for the symmetric groups is given using the residue content of a partition $\lambda$, \cite{JK}, Theorem 2.7.14. In particular for $\lambda,\mu\in \P_\reg(r)$ we have that  $D^\lambda$ and $D^\mu$ belong to the same block if and only if $\cont(\lambda)=\cont(\mu)$.

\bigskip

\q  Let $m$ and $n$ be positive integers.   If $V$ is a module for $\Sym_m\times \Sym_n$ over $k$ then   the space of fixed points $V^{\Sym_m}$ is naturally a $\Sym_n$-module.  We write $\FF_{m,n}$, or $\FF_m$ for short, for the corresponding functor from $\mod(\Sym_m\times \Sym_n)$ to $\mod(\Sym_n)$.   We write $R^i\FF_m$, $i\geq 0$,  for the derived functors.   We regard $\Sym_m\times \Sym_n$ as subgroup of $\Sym_{m+n}$, in the obvious way, with $\Sym_m$ permuting $\{1,2,\ldots,m\}$ and $\Sym_n$ permuting $\{m+1,\ldots,m+n\}$.   For a $\Sym_{m+n}$-module $V$ we write simply $\FF_m(V)$ for $\FF_m(\Res_{\Sym_m\times \Sym_n}(V))$ and $R^i\FF_m(V)$ for $R^i\FF_m(\Res_{\Sym_m\times \Sym_n}(V))$, \\
$i\geq 0$.   For a $\Sym_m$-module $V$ we also write $\FF_m(V)$ for the space of fixed points $V^{\Sym_m}$.   

\q To simplify the lengthy calculations that follow we write $F_m(V)$ for\\
 $\dim \FF_m(V)$, for $V\in \mod(\Sym_m\times\Sym_n)$, and for a partition  $\lambda$ of $r\geq m$ we write  $\FF_m(\lambda)$ for the $\Sym_{r-m}$-module  $\FF_m(\Sp(\lambda))$ and $F_m(\lambda)$ for $F_m(\Sp(\lambda))$. 

\bs

\q  When dealing with filtrations of modules for symmetric groups it will be convenient to take the terms in descending order.  Let  $V=V_1\geq  V_2\geq  \cdots\geq  V_{s+1}=0$ be a filtration of  a finite dimensional $\Sym_r$-module $V$. We call this a  a Specht filtration if for each $1\leq i\leq s$ the quotient $V_i/V_{i+1}$ is isomorphic to $\Sp(\lambda^i)$ for some $\lambda^i\in \P(r)$.  In general  if $V_i/V_{i+1}$ is isomorphic to $X_i$ (as a $\Sym_r$-module), for $1\leq i\leq s$ we write   $V\sim X_1 | X_2 | \ldots | X_s$. We shall also indicate such a filtration by the diagram
$$\begin{array}{ccc}
X_1\cr
\hrulefill\cr
X_2\cr
\hrulefill\cr
\cdots \cr
\hrulefill\cr
X_s.
\end{array}$$

\q  By \cite{James},  Theorem 9.3,  and the above  description of the blocks  we have the following result.

\bigskip

{\bf Theorem  1.3.5} \sl Let $\lambda\in \P(r)$.

(i) (Branching Theorem) The $\Sym_{r-1}$-module $\Res_{\Sym_{r-1}}(\Sp(\lambda))$ has a Specht filtration with sections $\Sp(\lambda_R)$, where $R$ runs over the set of removable nodes of $\lambda$ and in such a way that the module $\Sp(\lambda_R)$ occurs above $\Sp(\lambda_S)$  if $\lambda_R>\lambda_S$.

(ii) If the removable nodes of $\lambda$ have distinct residues  then we have 

$$\Res_{\Sym_{r-1}}(\Sp(\lambda))\cong \bigoplus_R\Sp(\lambda_R)$$

where $R$ runs over the set of removable nodes of $\lambda$.

\rm
\bs

The following bound (a consequence of the left exactness of  $\FF_m$) will be used many times in the rest of the paper.

\bs

{\bf Corollary 1.3.6}  \sl  Let $\lambda$ be a partition of $r$ and suppose $m<r$.   The we have
$$F_m(\lambda)\leq \sum_R F_m(\lambda_R)$$
where $R$ runs over removable nodes of $\lambda$, with equality if  the residues of the removable nodes are distinct.

\bs
\rm

{\bf Remark 1.3.7} \sl   Another useful bound  is obtained by reduction from characteristic $0$. Suppose $\lambda$ is a partition of $r\geq m$.  Then the Specht module $\Sp(\lambda)$ (over $k$) is defined over $\zed$, i.e.,  we have $\Sp(\lambda)=k\otimes _\zed \Sp_\zed(\lambda)$,  for a suitable $\zed$-form, $\Sp_\zed(\lambda)$,  of the Specht module $\Sp_\que(\lambda)$ over $\que$, for the rational group algebra $\que \Sym_r$.  Then the group of invariants $\Sp_\zed(\lambda)^{\Sigma_m}$ has rank equal to the dimension $\dim_\que \Sp_\que(\lambda)^{\Sym_m}$ and $\Sp_\zed(\lambda)/\Sp_\zed(\lambda)^{\Sym_m}$ is torsion free so  we have an embedding $k\otimes_\zed \Sp_\zed(\lambda)^{\Sym_m}\to \Sp(\lambda)^{\Sym_m}$. Hence we have  
$$\dim_\que \Sp_\que(\lambda)^{\Sym_m}\leq F_m(\lambda).$$

\rm 

\bs

\q It will be crucial for us to know which irreducible modules $D^\lambda$ of $\Sym_r$ remain irreducible on restriction to the subgroup $\Sym_{r-1}$. By a result of Kleshchev,  \cite{BK}, Theorem E',  we have the following description. 

\bs

{\bf Lemma 1.3.8} \sl Let $\lambda\in \P_\reg(r)$. The  $\Sym_{r-1}$-module $\Res_{\Sym_{r-1}}(D^\lambda)$ is irreducible if and only if $\lambda$ has a unique normal node. In this case $\Res_{\Sym_{r-1}}(D^\lambda)\cong D^{\lambda_R}$, where $R$ is the unique normal node of $\lambda$.

\rm
\bs


\q Let $\lambda \in \P(r)$. We recall some well known results by G. James and M. Peel, \cite{JP},  Theorem 3.1, for the restriction of a Specht module $\Sp(\lambda)$ on the Young subgroup $\Sym_m\times\Sym_{r-m}$ of $\Sym_r$. In order to do this we must discuss first skew-diagrams and the corresponding skew Specht modules.

\q Let $m$ be a positive integer with $m<r$. Let $\lambda\in \P(r)$ and $\mu\in \P(m)$ with $\lambda=(\lambda_1,\dots,\lambda_s)$ and $\mu=(\mu_1,\dots,\mu_t)$. Suppose that $\mu_i\leq\lambda_i$ for all $i$. Then the skew digram $\lambda/\mu$ is defined to be 

$$\{(i,j)\in \nat\times \nat    \,| \, 1\leq i\leq s, \mu_i <j\leq \lambda_i\}.$$

To each skew diagram there is  associated a skew Specht module $\Sp(\lambda/\mu)$. The  construction can be found in \cite{JP} where it is shown that $\Sp(\lambda/\mu)$ has a Specht filtration as a $\Sym_{r-m}$-module. More precisely we have the following result.

 \bs
 
 {\bf Lemma 1.3.9} \sl Let $\lambda\in \P(r)$ and $m$ a positive integer with $m<r$.\\
 (i) The Specht module $\Sp(\lambda)$ has a filtration as a  $\Sym_m\times\Sym_{r-m}$-module with sections isomorphic to $\Sp(\mu)\otimes \Sp(\lambda/\mu)$, where each partition $\mu$ of $m$ such that $\lambda/\mu$ exists occurs exactly once. Moreover the section $\Sp(\mu)\otimes \Sp(\lambda/\mu)$ occurs above $\Sp(\tau)\otimes \Sp(\lambda/\tau)$ in this filtration if $\mu>\tau$.\\
 (ii) The $\Sym_{r-m}$-module $\Sp(\lambda/\mu)$ has a Specht filtration and the Specht factors are given by the Littlewood-Richardson Rule. 
 Furthermore, the filtration may be chosen so that a section $\Sp(\alpha)$ occurs above a section $\Sp(\beta)$ if $\alpha>\beta$ (for partitions $\alpha,\beta$ of $r-m$).

 \rm

 \bigskip
 
 \q The assertion concerning the ordering of the terms of the filtration in (i) is not explicitly stated in \cite{JP} but is clear from the construction.  The assertion concerning the ordering of the terms of the filtration in (ii) follows from the corresponding result for representations of general linear groups, see e.g., \cite{Skew}, p477.
 
 \bigskip

 \q We assume of the reader some familiarity  with the  Littlewood-Richardson Rule, hereafter abbreviated to LRR.  More details can be found in \cite{James} 16.4.

 \q For a module $X$ of finite length we write ${\rm soc}(X)$ for its socle, ${\rm hd}(X)$ for its head and ${\rm rad}(X)$ for its radical.

 \subsection{Some Preliminary Results}

\q  It will be  useful in what follows  to know $\FF_p(\lambda)$ and  $R^1\FF_p(\lambda)$ for partitions $\lambda=(\lambda_1,\lambda_2,\lambda_3)$ with $\deg(\lambda)=p$ that lie in the principal block.
 
 \bigskip
 
{\bf Lemma 1.4.1} 
\sl
(i) For $p=2$ we have  $\FF_2(2)=\FF_2(1,1)=k.$\\
(ii) For $p\geq3$ we have $\FF_p(p)=\FF_p(p-1,1)=k$,  
$\FF_p(p-2,1,1)=0$ and $R^1\FF_p(p-2,1,1)=k$.

\rm
\bs

\begin{proof}

(i) For $p=2$ the modules $\Sp(2)$ and $\Sp(1,1)$ are trivial and one dimensional so the result is clear.

(ii)  Now assume $p\geq 3$. Certainly we have   $\FF_p(p)=k$.   By \cite{Peel},  Theorem 2,  we have that $\Sp(p-1,1)$ has exactly two composition factors and we have a short exact sequence,

$$0\rightarrow D^{(p)}\rightarrow \Sp(p-1,1)\rightarrow D^{(p-1,1)}\rightarrow0.$$

Applying   $\FF_p$ we obtain $\FF_p(p-1,1)=k$.

\q We consider now the Specht module $\Sp(p-2,1,1)$.  If $p=3$ then \\
$\Sp(1,1,1)\cong D^{(2,1)}\cong\sgn_3$ and so   $\FF_3(1,1,1)=0$. Moreover by \cite{KN},  Lemma 5.1 (ii), we get  $R^1\FF_3(1,1,1)=R^1\FF_3(D^{(2,1)})=k$. 

\q Assume now  $p\geq5$. By  \cite{Peel},  Theorem 2, the module  $\Sp(p-2,1,1)$ has two composition factors and we have  a short exact sequence

$$0\rightarrow D^{(p-1,1)}\rightarrow \Sp(p-2,1,1)\rightarrow D^{(p-2,1,1)}\rightarrow0.$$

Therefore we have $\FF_p(p-2,1,1)=0$. Moreover by \cite{KN},  Lemma 5.1 (ii),  we have  $R^1\FF_p(D^{(p-1,1)})=k$ and $R^1\FF_p(D^{(p-2,1,1)})=0$. Now the long exact sequence obtained by applying the derived functors of  $\FF_p$ gives  \\
 $R^1\FF_p(p-2,1,1)=R^1\FF_p(D^{(p-1,1)})=k$.

\end{proof}

\subsection{Methodology}

\q We here describe a  method that we will use in what follows  to calculate the dimension of spaces of invariants.

\q Let $m$ and $n$ be positive integers. We denote by $b_{m}$ the principal block idempotent of $k\Sym_m$. If $V$ is a $k(\Sym_m\times \Sym_n)$-module then $b_mV$ is a $k(\Sym_m\times \Sym_n)$-submodule (indeed a $k(\Sym_m\times \Sym_n)$-module summand) of $V$ and we have   $R^i\FF_{m,n}V=R^i\FF_{m,n} (b_mV)$, for $i\geq 0$. 

\q Let $p\geq 3$ and $\lambda=(\lambda_1,\lambda_2,\lambda_3)$ be a partition with $\deg(\lambda)=r>p$ $\lambda_1=p$ and $\lambda_3\neq0$. By Lemma 1.4.1 we have that the Specht module $\Sp(\lambda)$ has a filtration as a $\Sym_p\times \Sym_{r-p}$ with sections $\Sp(\mu)\otimes \Sp(\lambda/\mu)$, where each partition $\mu$ of $m$ such that $\lambda/\mu$ exists occurs exactly once. Moreover the section $\Sp(\mu)\otimes \Sp(\lambda/\mu)$ occurs above $\Sp(\tau)\otimes \Sp(\lambda/\tau)$ in this filtration if $\mu>\tau$.

\q  Now the partitions $\mu$ with at most $3$ parts and degree $p$  such that  $\Sp(\mu)$, lies in  the principal block of $\Sym_p$ are $p$, $(p-1,1)$ and $(p-2,1,1)$. Applying $b_{p}$ to $\Sp(\lambda)$ we obtain the 
 $\Sym_p\times \Sym_{r-p}$-module filtration 
$$\begin{array}{ccc}
\Sp(p)\otimes \Sp(\lambda/(p))\cr
\hrulefill\cr
\Sp(p-1,1)\otimes \Sp(\lambda/(p-1,1))\cr
\hrulefill\cr
\Sp(p-2,1,1)\otimes\Sp(\lambda/(p-2,1,1))\cr
\end{array}$$ 
of $b_p\Sp(\lambda)$.  We shall call this the {\it standard filtration of the principal $\Sym_p$-block component},   or just the {\it principal $\Sym_p$-block filtration}, 
of $\Sp(\lambda)$.

\q Thus we have $R^i\FF_p(\Sp(\lambda))=R^i\FF_p(b_p\Sp(\lambda))$, $i\geq 0$,  and  $b_p\Sp(\lambda)$ has the $\Sym_p\times \Sym_{r-p}$-module filtration depicted above.

\q In analysing $R^i\FF_p(\Sp(\lambda))$ it will  often be convenient to first deal with the module $X$, say, corresponding to the   bottom two sections of the above filtration. Thus we have an exact sequence 

\begin{align}
0&\rightarrow  \Sp(p-2,1,1)\otimes \Sp(\lambda/(p-2,1,1))\rightarrow X\cr
&\rightarrow\Sp(p-1,1)\otimes\Sp(\lambda/(p-1,1))\rightarrow 0.
\end{align}

Let  $Y=\FF_p(X)$.  Then   $Y\leq \FF_p(\lambda)$. The structure of the $\Sym_{r-p}$-module $Y$ will be the key point  in order to calculate the dimensions of the module of  invariants $\FF_p(\lambda)$. 
 
 \q We will analyse $Y$ in the following way.  Using the above short exact sequence of $X$ we get an exact sequence
\begin{align*}
0&\rightarrow H^0(\Sym_p,\Sp(p-2,1,1)\otimes\Sp(\lambda/(p-2,1,1)))\rightarrow Y \cr
&\rightarrow H^0(\Sym_p,\Sp(p-1,1)\otimes \Sp(\lambda/(p-1,1)))\cr
&\rightarrow H^1(\Sym_p,\Sp(p-2,1,1)\otimes\Sp(\lambda/(p-2,1,1))).
\end{align*}
Using now Lemma 1.4.1 we obtain the exact sequence
 
\begin{align}0\rightarrow Y\rightarrow \Sp(\lambda/(p-1,1))\rightarrow \Sp(\lambda/(p-2,1,1)).\end{align}

 This sequence, used in conjunction  with  LRR, will be our main tool in the analysis of $Y$.

\bs\bs\bs\bs


\section{Dimensions  of Invariants.}

\q We calculate, for  $p\geq 5$, the dimension of the space of    $\Sym_p$-invariants of Specht modules labelled by partitions with at most three parts and first part at most $p$. However, for the moment $p$ is  arbitrary.

 \begin{lemma}
 
 For a partition $\lambda$ with $\lambda_1\leq p-1$ and $\deg(\lambda)\geq p$ we have,
 
 (i) $\FF_p(\lambda)\neq 0$ if and only if $\lambda$ has the form $(p-1)^la$, for some $l\geq 1$ and $0\leq a<p-1$;
 
 (ii)  $\FF_p(\lambda)=k$, the trivial $\Sym_{r-p}$-module, for $\lambda=(p-1)^la$, $l\geq 1$, $0\leq a < p-1$.
 
 \end{lemma}

 \begin{proof}

 For $p=2$  we have that  $\Sp(1^r)\cong\sgn_r\cong \Sp(r)$ is the trivial module of $\Sym_r$ and so the result  is clear. 

\q Therefore we may assume that $p\geq3$.    Next note that if $\deg(\lambda)=p$ then $\FF_p(\lambda)\neq 0$  implies that $\lambda=(p-1,1)$ e.g., by \cite{James},  24.4 Theorem and we then get $\FF_p(\lambda)=k$, e.g., by \cite{Peel}, Theorem 2

\q We now proceed by induction on the degree of $\lambda$.  We write  $r=l(p-1)+a$ for integers $r,a$ with $0\leq a < p-1$. 

\q Note that  
$\FF_r((p-1)^la)\neq 0$. This follows directly by \cite{James}, Theorem 24.4. Hence  $\FF_p((p-1)^la)$ contains a copy of the trivial $\Sym_{r-p}$-module $k$.

\q  If $\lambda=(p-1)^l$ then we have  $\Res_{\Sym_{r-1}}(\Sp(\lambda))=\Sp((p-1)^{l-1}(p-2))$.  By induction  we have  $F_p((p-1)^{l-1}(p-2))$=1 and since $F_p(\Sp((p-1)^l))=F_p(\Sp((p-1)^{l-1}(p-2)))$ we have $\FF_p(\lambda)=k$.  Likewise for $\lambda=(p-1)^la$ with $0<a<p-1$ we have, by Corollary 1.3.6, 
$$F_p(\lambda)\leq F_p((p-1)^l(p-2)a)+F_p((p-1)^l(a-1))\leq 0+1=1$$
and $\FF_p(\lambda)=k$.

\q Now assume that $\lambda$ is as in the statement of the Lemma and that $\FF_p(\lambda)\neq 0$.   By Corollary 1.3.6  and induction there exists a removable node $R$ of $\lambda$ such that $\lambda_R=(p-1)^lb$, for some $l>0$, $0\leq b< p-1$. Hence $\lambda=(p-1)^l(b+1)$ or $\lambda=(p-1)^lb1$. 

\q We now show that $\lambda\neq (p-1)^lb1$.  If $\lambda=(p-1)^lb1$ then, by Corollary 1.3.6, left exactness of $\FF_p$ and induction, we have $F_p(\lambda)=1$ and indeed $\FF_p(\lambda)=k$.

\q  Therefore $\Hom_{\Sym_p\times \Sym_{r-p}}(k,\Sp(\lambda))\neq0$ and so, by Frobenius reciprocity, 
 $\Hom_{\Sym_r}(M(p,r-p),\Sp(\lambda))\neq0$. Since $\lambda$ is restricted  we have that $\Sp(\lambda)$ has simple socle $f(L(\lambda))$ (see \cite{EGS}, 6.4). Therefore $\Hom_{\Sym_r}(M(p,r-p),\Sp(\lambda))\neq0$ gives that $[M(p,r-p):f(L(\lambda))]\neq0$. Moreover we have that
 
  \begin{align*}
  [S^pE\otimes S^{r-p}E:L(\lambda)]&=[f(S^pE\otimes S^{r-p}E):f(L(\lambda))]\cr
  &=[M(p,r-p):f(L(\lambda))]
\end{align*}
 
  and so $L(\lambda)$ occurs as a composition factor of  $S^pE\otimes S^{r-p}E$. In particular using the language of  \cite{DG3} (see also \cite{DG4}) we have that $\lambda$ is a $2$-good partition and since it is restricted $\lambda$ is a $2$-special partition. Hence by  \cite{DG3},  Proposition 4.10 or by \cite{DG4}, Lemma 2.1 we get that $\lambda=(p-1)^ma+(p-1)^nc$ for some $m,n\geq0$ and $0\leq a,c<p-1$ or $\lambda=(p-2)^mac$ for $m\geq0$ and $0\leq a,c<p-2$. Since $\lambda_1=p-1$ we conclude   that $\lambda=(p-1)^ma$  and we have a contradiction. The proof now is complete.

 \end{proof}

\q  From now on  we assume that $p\geq 5$.  (We treat the cases $p=2$ and $p=3$ separately  in the last two sections.)
 We calculate the dimension of the space of invariants $\FF_p(\lambda)$ for a partition $\lambda=(\lambda_1,\lambda_2,\lambda_3)$ with at most $3$ parts and $\lambda_1=p$.  We will do this in several steps.
 
 \begin{lemma}
 
 Let $\lambda=(p)$ then $F_p(\lambda)=1$.
 
 \end{lemma}
 
 \begin{proof}
Clear. 
 \end{proof}
 
 \begin{lemma} For $0<a\leq p$ we have
 $$F_p(p,a)=
 \begin{cases} 
 a, &\mbox{if } 1\leq a\leq p-1; \cr
p-1, & \mbox{if } a=p. 
\end{cases}$$ 
 
 \end{lemma}
 
 \begin{proof}   First take $a=1$. We have 
 \begin{align*}\FF_p(p,1)&=\Hom_{\Sym_p}(k,\Sp(p,1))=\Hom_{\Sym_{p+1}}(M(p,1),\Sp(p,1))
 \end{align*}
 by reciprocity. The partition $(p,1)$ is a core so  the permutation module $M(p,1)$ decomposes as $M(p,1)\cong \Sp(p,1)\oplus\Sp(p+1)$. Therefore we get directly that 
 
$$\Hom_{\Sym_{p+1}}(M(p,1),\Sp(p,1))=\Hom_{\Sym_{p+1}}(\Sp(p,1),\Sp(p,1))=k$$
 
   and we are done.

\q Let now $2\leq a\leq p-1$. The partition $(p,a)$ has two removable nodes $(1,p)$ and $(2,a)$  and these have distinct residues. Therefore by Corollary 1.3.6  we have 
 $F_p(p,a)=F_p(p-1,a)+F_p(p,a-1)$.   By Lemma 2.1 and induction on $a$ we get  
  $F_p(p,a)=1+a-1=a$.

\q Finally for   $a=p$ we have $F_p(p,p)=F_p(p,p-1)$, by Corollary 1.3.6 and this is $p-1$, as  proved already.

 \end{proof}

 \begin{lemma}  For $\lambda=(p,a,1)$ with $1\leq a\leq p$ we have 
  $$F_p(p,a,1)=
 \begin{cases} 
a(a+1)/2, &\mbox{if } a\neq p-1; \cr
p(p-1)/2+1, & \mbox{if }  a=p-1 .
\end{cases} $$
\end{lemma}

\begin{proof}

First let  $a=1$. Then the partition $\lambda=(p,1,1)$ has two removable nodes $(1,p)$ and $(3,1)$ with residues $p-1$ and $p-2$ respectively. Hence, by Corollary 1.3.6  we have 
 $F_p(p,1,1)=F_p(p-1,1,1)+F_p(p,1)$. It follows by Lemmas  2.1 and  2.3 that $F_p(p,1,1)=0+1=1$, as required.

\q Let now $1<a\leq p-2$. The partition $(p,a,1)$ has three removable nodes $(1,p)$, $(2,a)$ and $(3,1)$ with distinct residues $p-1$, $a-2$ and $p-2$ . Thus by Corollary 1.3.6 we have 
$$F_p(p,a,1)=F_p(p-1,a,1)+F_p(p,a-1,1)+F_p(p,a).$$
 Using now induction on $a$ and Lemmas 2.1 and 2.3 we conclude that 
 $$F_p(p,a,1)=(a-1)a/2+a=a(a+1)/2.$$

\q Let $a=p-1$. The partition $(p,p-1,1)$ has three removable nodes $(1,p)$, $(2,p-1)$ and $(3,1)$ with residues $p-1$, $p-3$ and $p-2$ respectively. Hence we get  
$$F_p(p,p-1,1)=F_p(p-1,p-1,1)+F_p(p,p-2,1)+F_p(p,p-1).$$ 
 By the result of the previous paragraph and Lemmas 2.1 and 2.3 we get that 
 $$F_p(p,p-1,1)=1+(p-2)(p-1)/2+p-1=p(p-1)/2+1.$$

\q Finally we have the case in which $a=p$. By Corollary 1.3.6 we have 
$$F_p(p,p,1)\leq F_p(p,p-1,1)+F_p(p,p).$$
  Therefore by Lemma 2.3 and the previous step we have the upper bound  $F_p(p,p,1)\leq p(p+1)/2$. 

\q We show that in fact $F_p(p,p,1)=p(p+1)/2$. We consider the principal block filtration 
$$\begin{array}{ccc}
\Sp(p)\otimes U\cr
\hrulefill\cr
\Sp(p-1,1)\otimes V\cr
\hrulefill\cr
\Sp(p-2,1,1)\otimes W\cr
\end{array}$$
 of $\Sp(p,p,1)$. By LRR we have 
 $U\sim \Sp(p,1)$, \\ 
 $V\sim \Sp(p,1)|\Sp(p-1,2)|\Sp(p-1,1,1)$ and 
 $W\sim \Sp(p-1,2)$.

\q We observe at this point that $\core(p,1)=(p,1)$, $\core(p-1,2)=(1)$ and $\core(p-1,1,1)=(p-1,1,1)$ and since these are all different we get \\
 $V=\Sp(p,1) \oplus \Sp(p-1,2)\oplus \Sp(p-1,1,1)$.

\q Let now $X$ be  the $\Sym_p\times \Sym_{p+1}$-module corresponding to  the bottom two sections of the above filtration. Thus we have a short exact sequence 
\begin{align*}
0&\rightarrow  \Sp(p-2,1,1)\otimes\Sp(p-1,2)\rightarrow X\cr
&\rightarrow\Sp(p-1,1)\otimes(\Sp(p,1) \oplus \Sp(p-1,2)\oplus \Sp(p-1,1,1))\rightarrow 0.
\end{align*}
Let $Y=\FF_p(X)$. Thus  $Y\leq \FF_p(p,p,1)$. Applying $\FF_p$ to  this exact sequence and using Lemma 1.4.1 we get an exact sequence of  $\Sym_{p+1}$-modules,
$$0\rightarrow Y\rightarrow \Sp(p,1)\oplus \Sp(p-1,2)\oplus \Sp(p-1,1,1)\rightarrow\\
\Sp(p-1,2).$$
By block considerations we have  $\Hom_{\Sym_{p+1}}(\Sp(p,1),\Sp(p-1,2))=0$ and $\Hom_{\Sym_{p+1}}(\Sp(p-1,1,1),\Sp(p-1,2))=0$. 
 Therefore  $\Sp(p,1)\oplus \Sp(p-1,1,1)$ embeds in $Y$.   Using the hook  formula for the dimension of Specht modules we have that $\dim(p,1)=p$ and $\dim(p-1,1,1)=p(p-1)/2$. Hence $\dim Y\geq p(p+1)/2$. So in fact  $Y=\FF_p(p,p,1)$ and $F_p(p,p,1)=p(p+1)/2$.

\end{proof}

\q The next step is to calculate   $F_p(p,a,2)$ for $2\leq a\leq p$. The proof will be done in several steps.  We will treat the cases $a=2,3$ separately before we give a general proof for $a\geq 4$. We will then gather these results together and   give  a uniform  formula for   $F_p(p,a,2)$,  for  $2\leq a\leq p$.

\begin{lemma} We have $F_p(p,2,2)=\dim(2,2)$.

\end{lemma}

\begin{proof} By Corollary 1.3.6 we have
$$F_p(p,2,2)\leq F_p(p-1,2,2)+F_p(p,2,1)$$
 and by Lemmas  2.1 and 2.4 we get               
$F_p(p,2,2)\leq3$.

\q Now we consider the principal block $\Sym_p\times\Sym_4$-filtration of $\Sp(p,2,2)$. This has the form
$$\begin{array}{ccc}
\Sp(p)\otimes U\cr
\hrulefill\cr
\Sp(p-1,1)\otimes V\cr
\hrulefill\cr
\Sp(p-2,1,1)\otimes W\cr
\end{array}$$
with $U\sim \Sp(2,2)$, $V\sim \Sp(3,1)|\Sp(2,2)|\Sp(2,1,1)$ and\\
 $W\sim \Sp(3,1)|\Sp(2,1,1)$.

\q Since $p\geq 5$ the Specht modules $\Sp(3,1),\Sp(2,2)$ and $\Sp(2,1,1)$ are simple $\Sym_4$-modules. Furthermore, we  have  $\dim(3,1)=3$, $\dim(2,2)=2$ and $\dim(2,1,1)=3$.  

\q  Let  $X$ be the $\Sym_p\times \Sym_{p+1}$-module corresponding to the   bottom two sections of the above filtration. Thus we have an exact sequence 
\begin{align*}
0&\rightarrow  \Sp(p-2,1,1)\otimes(\Sp((3,1))\oplus \Sp(2,1,1))\rightarrow X\cr
&\rightarrow\Sp(p-1,1)\otimes(\Sp(3,1)\oplus\Sp(2,2)\oplus\Sp(2,1,1))\rightarrow 0.
\end{align*}
Let $Y=\FF_p(X)$. Then, $Y\leq \FF_p(p,2,2)$. By Lemma 1.4.1 we get an exact sequence,
  \begin{align*}
&0\rightarrow Y\rightarrow \Sp(3,1)\oplus\Sp(2,2)\oplus\Sp(2,1,1) \cr
&\rightarrow\Sp(3,1)\oplus \Sp(2,1,1).
\end{align*}
Since $\Hom_{\Sym_4}(\Sp(2,2),\Sp(3,1))=\Hom_{\Sym_4}(\Sp(2,2),\Sp(2,1,1))=0$ we get  that $\Sp(2,2)$ embeds in  $Y$.  Moreover, since $\dim Y\leq3 $ we conclude that $Y=\Sp(2,2)$.

\q Now we consider the short  exact sequence
$$0\rightarrow X \rightarrow \Sp(p,2,2)\rightarrow \Sp(p)\otimes \Sp(2,2)\rightarrow 0.$$
Applying $\FF_p$,  we get an  exact sequence
$$0\rightarrow Y \rightarrow \FF_p(p,2,2)\rightarrow \Sp(2,2).$$

Since $F_p(p,2,2)\leq 3$ and $\Sp(2,2)$ is an irreducible $\Sym_4$-module of dimension $2$ we have  $\FF_p(p,2,2)=Y=\Sp(2,2)$ and  we are done.

\end{proof}

\begin{lemma} We have $F_p(p,3,2)=\dim(3,2)$.
\end{lemma}

\begin{proof}  By Corollary 1.3.6 we have 
$$F_p(p,3,2)\leq F_p(p-1,3,2)+F_p(p,2,2)+F_p(p,3,1)$$
 and from  Lemmas 2.1, 2.4 and 2.5 we get $F_p(p,3,2)\leq 8$.

\q Now we consider the principal block $\Sym_p\times \Sym_5$-filtration 
$$\begin{array}{ccc}
\Sp(p)\otimes U\cr
\hrulefill\cr
\Sp(p-1,1)\otimes V\cr
\hrulefill\cr
\Sp(p-2,1,1)\otimes W\cr
\end{array}$$
of $\Sp(p,3,2)$. From LRR we have
$U\sim \Sp(3,2)$, \\
$V\sim \Sp(4,1)|\Sp(3,2)|\Sp(3,2)|\Sp(3,1,1)|\Sp(2,2,1)$ and \\
 $W\sim \Sp(4,1)|\Sp(3,2)|\Sp(3,1,1)|\Sp(2,2,1)$. Using the hook  formula for dimensions we have that $\dim(4,1)=4$, $\dim(3,2)=5$,    $\dim(3,1,1)=6$ and $\dim(2,2,1)=5$.  
  
\q Let  $X$ be the $\Sym_p\times \Sym_{5}$-module corresponding to  the bottom two  sections of the above filtration and  $Y=\FF_p(X)$. Then $Y\leq \FF_p(p,3,2)$. By Lemma 1.4.1 we get an exact sequence
 $$0\rightarrow Y\rightarrow V\rightarrow W.$$

Since the partition $(3,2)$ is a $p$-core we see, by block considerations, that the kernel of the map $V\to W$ contains a copy of $\Sp(3,2)$.
Since $\dim(3,2)=5$ and $\dim Y\leq 8$ we can write $Y=Y_0\oplus {\bar Y}$, where ${\bar Y}$ has dimension at most $3$.  Moreover, by block considerations, we can write $V=V_0\oplus V_1\oplus {\bar V}$ and $W=W_0\oplus W_1\oplus {\bar W}$,  with $V_0=\Sp(3,2)\oplus \Sp(3,2)$, $V_1=\Sp(2,2,1)$, $W_0=\Sp(3,2)$, $W_1=\Sp(2,2,1)$.  Moreover, there are   short exact sequences of $S_5$-modules $0\to A\to {\bar V}\to B\to 0$, $0\to C\to {\bar W}\to D\to 0$ with $A=C=\Sp(3,1,1)$, $B=D=\Sp(4,1)$.  Further, since $\dim {\bar Y}\leq 3$ and $\Sp(3,2)$ and $\Sp(2,2,1)$ are $5$-dimensional modules the map $Y\to V$ takes ${\bar Y}$ into ${\bar V}$ and we get an exact sequence $0\to {\bar Y}\to {\bar V}\to {\bar W}$.  Now the composite  $A\to {\bar W}\to D$ is $0$ by Lemma 1.3.2(ii) so that the map ${\bar V}\to {\bar W}$ restricts  to give a map $A\to C$. By Lemma 1.3.2 (iii) this is either $0$ or an isomorphism. But if this is the $0$ map then ${\bar Y}$ contains a copy of $\Sp(3,1,1)$. But this is impossible since $\dim {\bar Y}\leq 3$. Hence the map $V\to W$ induces an isomorphism $A\to C$ and we get an induced map ${\bar V}/A\to {\bar W}/C$. Once more, by Lemma 1.3.2 (iii) we find that this must be an isomorphism for otherwise we would have that the map ${\bar V}\to {\bar W}$ would have image $C$ and kernel of dimension greater than $3$. The map  ${\bar V}\to {\bar W}$ is an isomorphism and ${\bar Y}=0$ and so $Y=\Sp(3,2)$ is $5$ dimensional.

\q Now we consider the short  exact sequence
$$0\rightarrow X \rightarrow \Sp(p,3,2)\rightarrow \Sp(p)\otimes \Sp(3,2)\rightarrow 0$$
of $\Sym_p\times \Sym_5$-modules.  Applying $\FF_p$,  we get the exact sequence
$$0\rightarrow Y \rightarrow \FF_p(p,3,2)\rightarrow \Sp(3,2).$$ 
Since $\Sp(3,2)$ is irreducible we have that either $\FF_p(p,3,2)=Y=\Sp(3,2)$ or $\FF_p(p,3,2)=\Sp(3,2)\oplus\Sp(3,2)$. But $F_p(p,3,2)\leq 8 $ and so the latter  is impossible. Hence $X=Y=\Sp(3,2)$ and so we are done.
\end{proof}

\q Our next target is to calculate the  $F_p(p,a,2)$ for $a\leq4\leq p-2$.   Perhaps surprisingly, we find it convenient to first treat  modules of the form $\Sp(p,a,3)$ and then use this knowledge to deal with those of the form $\Sp(p,a,2)$.

\begin{lemma} For $3\leq a\leq p-1$ we have
$$
F_p(p,a,3)=
 \begin{cases} 
 \sum\limits_{i=3}^a F_p(p,i,2), &\mbox{if } a\leq p-2;\cr
\sum\limits_{i=3}^{p-1}F_p(p,i,2)+1, & \mbox{if } a=p-1 .
\end{cases} $$
\end{lemma}

\begin{proof}
 We prove the statement by induction on $a$. For $a=3$ the partition $(p,3,3)$ has the two removable nodes $(1,p)$ and $(3,3)$ with residues $p-1$ and $0$ respectively. Thus, by Corollary 1.3.6 we have 
 $$
 F_p(p,3,3)=F_p(p-1,3,3)+ F_p(p,3,2).
 $$
  By Lemma 2.1 we have that $F_p(p-1,3,3)=0$ and so  $F_p(p,3,3)=F_p(p,3,2)$,  as required.
 
 \q Now suppose  $3<a\leq p-1$. The partition $(p,a,3)$ has three removable nodes $(1,p)$,  $(2,a)$ and $(3,3)$ with the distinct residues $p-1$, $a-2$ and $0$ respectively.  Thus, by Corollary 1.3.6 we have 
 $$F_p(p,a,3)=F_p(p-1,a,3)+F_p(p,a-1,3)+F_p(p,a,2).$$
   By Lemma 2.1 and induction on $a$ we immediately get  the result.
\end{proof}

\begin{lemma} For $3\leq a\leq p$ we have
$$\dim(a,3)=\sum\limits_{i=3}^a\dim(i,2).$$
\end{lemma}

\begin{proof}
This follows from the branching rule and induction.

\end{proof}

\begin{lemma} For $2\leq a\leq p-2$  we have 
$$F_p(p,a,2)=\dim(a,2).$$
 \end{lemma}

 \begin{proof} The statement is true for $a=2$ and $a=3$ by   Lemmas 2.5 and 2.6.  So now suppose  $4\leq a\leq p-2$ and write $r=\deg(\lambda)=p+a+2$.  By Corollary 1.3.6 we have 
$$
F_p(p,a,2)\leq F_p(p-1,a,2)+F_p(p,a-1,2) +F_p(p,a,1).
$$
 By Lemma 2.1, Lemma 2.2 and induction on $a$ we get  
 $$F_p(p,a,2)\leq\dim(a-1,2)+a(a+1)/2=a^2-1.$$

\q We will find $F_p(p,a,2)$ by calculating the dimension $F_p(p,a,3)$ and using induction. More precisely, by Lemmas  2.7,  2.8 and induction on $a$ we have 
\begin{align*}
F_p(p,a,2)&=F_p(p,a,3)-\sum_{i=3}^{a-1}F_p(p,i,2)\cr
&=F_p(p,a,3)-\sum_{i=3}^{a-1}\dim(i,2)\ \ \ \ \ \ \ \ \ \ \ (\dagger)\cr
&=F_p(p,a,3)-\dim(a-1,3).
\end{align*}
Since $F_p(p,a,2)\leq a^2-1$ we have the  upper bound  
$$F_p(p,a,3)\leq\dim(a-1,3)+a^2-1.$$
\q Consider now the $\Sym_p\times\Sym_{a+3}$-filtration of $\Sp(p,a,3)$. This has the form
$$\begin{array}{ccc}
\Sp(p)\otimes U\cr
\hrulefill\cr
\Sp(p-1,1)\otimes V\cr
\hrulefill\cr
\Sp(p-2,1,1)\otimes  W\cr
\end{array}$$
with $U\sim \Sp(a,3)$, 
$$V\sim \Sp(a+1,2) |\Sp(a,3)| \Sp(a,3)| \Sp(a,2,1)|\Sp(a-1,4)|\Sp(a-1,3,1)$$
 and 
 $$W\sim \Sp(a+1,2) | \Sp(a,3)|  \Sp(a,2,1)|\Sp(a-1,4)| \Sp(a-1,3,1)| \Sp(a-1,2,2)$$
(where  the section $\Sp(a-1,4)$ of $V$ and $W$ is omitted when $a=4$.)

\q Let  $X$ be the $\Sym_p\times \Sym_{a+3}$-module corresponding to  the  bottom  two sections of the above filtration. Thus   we have a short exact sequence,
$$0\rightarrow \Sp(p-2,1,1)\otimes  W\rightarrow X\rightarrow\Sp(p-1,1)\otimes V\rightarrow 0.$$
Let  $Y=\FF_p(X)$. Then, $Y \leq \FF_p(p,a,3)$. By Lemma 1.4.1 we get an exact sequence
$$0\rightarrow Y\rightarrow V\to W\eqno{(*)}.$$
We consider the map $V\to W$. The partition $(a,3)$ is a $p$-core and so by block considerations  one copy of the Specht module $\Sp(a,3)$ of $V$ must be in the kernel and we may write $Y=\Sp(a,3)\oplus {\bar Y}$, for some submodule ${\bar Y}$ of $Y$. We have 
\begin{align*}\dim {\bar Y}=
\dim Y -\dim(a,3)&\leq (a^2-1)+\dim(a-1,3)-\dim(a,3)\cr
&=(a^2-1)-\dim(a,2)\cr
&=a(a-1)/2\cr
&=\dim(a-1,1,1).
\end{align*}

\q It is an exercise in the use of the hook formula (which we leave to the reader) to check that $\dim \Sp(\lambda)> a(a-1)/2$ for all sections $\Sp(\lambda)$ of $V$, i.e., for $\lambda\in S=\{(a+1,2),(a,3),(a,2,1),(a-1,4), (a-1,3,1)\}$.  We have $a\leq p-2$ so that $a+3\leq p+1$. If $a+3\leq p$ then all elements of $S$ are $p$-cores so that $V$ is a direct sum of  simple modules $\Sp(\lambda)$, $\lambda\in S$.  Now any composition factor of ${\bar Y}$ is a composition factor of $V$ and since $\dim \Sp(\lambda)> \dim {\bar Y}$, for all $\lambda\in S$, we must have ${\bar Y}=0$. 

\q We now consider the case $a=p-2$. Then, for $\lambda\in S$,  we have that $\lambda$ is a $p$-core except for $\lambda= (p-1,2),(p-2,2,1)$. 
Hence any composition factor of ${\bar Y}$ must be a composition factor of $\Sp(p-1,2)$ or $\Sp(p-2,2,1)$.  Taking block components in $(*)$ we obtain a short exact sequence 
$$0\to {\bar Y} \to {\bar V}\to {\bar W}$$
where ${\bar V} \sim \Sp(p-1,2) \vert \Sp(p-2,2,1)$ and ${\bar W}\sim \Sp(p-1,2) \vert \Sp(p-2,2,1)$.  Arguing as in the proof  of Lemma 2.6 we obtain that ${\bar V}\to {\bar W}$ is an isomorphism and ${\bar Y}=0$.   Hence (in all cases) we have $Y=\Sp(a,3)$.

\q We  consider the exact sequence
$$0\rightarrow X\rightarrow \Sp(p,a,3)\rightarrow \Sp(p)\otimes U\rightarrow 0.$$
Applying  $\FF_p$  we get the exact sequence,
$$0\rightarrow Y \rightarrow \FF_p(p,a,3)\rightarrow \Sp(a,3).$$ 
Since  $(a,3)$ is a core   either $\FF_p(p,a,3)=Y=\Sp(a,3)$ or $\FF_p(p,a,3)=\Sp(a,3)\oplus\Sp(a,3)$. But 
$$F_p(p,a,3)-\dim Y \leq \dim(a-1,1,1)$$
 and $\dim(a,3)> \dim(a-1,1,1)$. Hence $X=Y=\Sp(a,3)$.

\q Finally from  the relation $(\dagger)$ we get  
\begin{align*}
F_p(p,a,2)&=F_p(p,a,3)-\dim(a-1,3)\cr
&=\dim(a,3)-\dim(a-1,3)\cr
&=\dim(a,2)
\end{align*}
and  we are done.

\end{proof}


\begin{lemma}

We have $F_p(p,p-1,2)=\dim(p-1,2)+\dim(p,1)$.

\end{lemma}

\begin{proof}  By Corollary 1.3.6 we have 
$$
F_p(p,p-1,2)\leq F_p(p-1,p-1,2)+F_p(p,p-2,2)+F_p(p,p-1,1).
$$
From  Lemmas 2.1 and 2.4 and  2.9 we get 
 \begin{align*}
 F_p(p,p-1,2)&\leq\dim(p-2,2)+p(p-1)/2+2\cr
 &=\dim(p-2,2)+\dim(p-1,1,1)+2  {\hskip 50pt (\dagger)}.
 \end{align*}
\q Consider now the principal block $\Sym_p\times\Sym_{p+1}$-filtration of $\Sp(p,p-1,2)$. This has the form  $$\begin{array}{ccc}
\Sp(p)\otimes U\cr
\hrulefill\cr
\Sp(p-1,1)\otimes V\cr
\hrulefill\cr
\Sp(p-2,1,1)\otimes  W\cr
\end{array}$$
with $U\sim \Sp(p-1,2)$,  with 
$$V\sim \Sp(p,1)|\Sp(p-1,2)|\Sp(p-1,2)|\Sp(p-1,1,1) |\Sp(p-2,3)| \Sp(p-2,2,1)$$
  and  with 
  $$W\sim \Sp(p-1,2)|\Sp(p-1,1,1)|\Sp(p-2,3)|\Sp(p-2,2,1).$$

 \q Let  $X$ be the $\Sym_p\times \Sym_{p+1}$-module corresponding to  the bottom two  sections of the above filtration. Thus we have the short exact sequence
$$0\rightarrow \Sp(p-2,1,1)\otimes  W\rightarrow X\rightarrow\Sp(p-1,1)\otimes V\rightarrow 0.$$
Let  $Y=\FF_p(X)$. Then  $Y\leq\FF_p(p,p-1,2)$. By Lemma 1.4.1 we get an exact sequence
 $$0\rightarrow Y\rightarrow V {\rightarrow} W.$$
We consider the map $V\to W$. The partition $(p,1)$ is a $p$-core and so by block considerations the copy of the Specht module $\Sp(p,1)$ of $V$ must be in the kernel of this map and so inside $Y$.  Hence we may write $Y=\Sp(p,1)\oplus {\bar Y}$, for some some submodule ${\bar Y}$ of $Y$. 

\q Consider now the block components ${\bar Y}_0$, ${\bar V}$, ${\bar W}$ of $Y,V,W$ for the block containing $(p-1,2)$ (and $(p-2,2,1)$).  We  have the exact sequence
$$0\to {\bar Y}_0\to {\bar V}\to {\bar W}.$$
We have 
${\bar V}\sim  \Sp(p-1,2)|\Sp(p-1,2)| \Sp(p-2,2,1)$ and\\
 ${\bar W}\sim \Sp(p-1,2)| \Sp(p-2,2,1).$  In particular ${\bar V}$ contains a copy $A$, say,  of $\Sp(p-2,2,1)$ and ${\bar W}$ contains  copy $B$, say, of $\Sp(p-2,2,1)$ with ${\bar V}/A$ and ${\bar W}/B$ filtered by copies of $\Sp(p-1,2)$. By Lemma 1.3.2(ii) the composite  map $A\to {\bar W}/B$ is zero, so the map ${\bar V}\to {\bar W}$ takes $A$ to $B$. The restriction $A\to B$ is either $0$ or an isomorphism. However, if this map is zero we would have a copy of $\Sp(p-2,2,1)$ in  ${\bar Y}_0$ and hence in $Y$. In this case we would have an embedding of  $\Sp(p,1)\oplus \Sp(p-2,2,1)$ in $Y$ and hence
$$\dim(p,1)+\dim(p-2,2,1)\leq \dim(p-2,2)+\dim(p-1,1,1)+2.$$
We show that this is not correct however,  by comparing these dimensions. In particular,  we have
  \begin{align*}
 \dim(p,1)+\dim(p-2,2,1)=1&+\dim(p-1,1)+\dim(p-3,2,1)\cr
 &+\dim(p-2,1,1)+\dim(p-2,2).
 \end{align*}
 Moreover, we have 
 \begin{align*}
 \dim(p-2,2)+\dim(p-1,1,1)+2=&\dim(p-2,2)+\dim(p-2,1,1)\cr
 &+\dim(p-1,1)+2
 \end{align*}
 and so 
  \begin{align*}
 \dim(p,1)+\dim(p-2,2,1)&-(\dim(p-2,2)+\dim(p-1,1,1)+2)\cr
 &=\dim(p-3,2,1)-1>0.
\end{align*}

\q By similar consideration we get that the map ${\bar V}/A\to {\bar W}/B$ is either $0$ or surjective. If this map is zero the the image of ${\bar V}\to {\bar W}$ is $B$ and the dimension of the kernel is $2\dim(p-1,2)$. In this case we would have $\dim Y\geq \dim(p,1)+2\dim(p-1,2)$. We leave it to the reader to check that in fact  $\dim(p,1)+2\dim(p-1,2)$ exceeds the upper bound for $F_p(p,p-1,2)$ (and hence for $\dim Y$) given in ($\dagger$). Hence the map ${\bar V}\to {\bar W}$ is surjective. Since ${\bar V} \sim \Sp(p-1,2)|\Sp(p-1,2)| \Sp(p-2,2,1)$ and  ${\bar W}\sim \Sp(p-1,2)| \Sp(p-2,2,1)$ the composition factors of the  kernel $K$, say, of the map ${\bar V}\to {\bar W}$ are those of $\Sp(p-1,2)$.  Thus $K$ has composition factors $D^{(p-1,2)}, D^{(p+1)}$ (each occurring once).  Moreover, it is easy to check that $\Ext^1_{\Sym_{p+1}}(D^{(p+1)},D^{(p-1,2)})=k$. It follows that $K$ is either $\Sp(p-1,2)$ or has a submodule isomorphic to $D^{(p-1,2)}$.  However, we have $\Hom_{\Sym_{p+1}}(D^{(p-1,2)}, \Sp(p-1,2))=0$ and ${\bar V}/A$ is filtered by copies of $\Sp(p-1,2)$. It follows that any copy of $D^{(p-1,2)}$ is contained in $A$. But the map ${\bar V}\to {\bar W}$ maps $A$ isomorphically to $B$ so that no copy of $D^{(p-1,2)}$ can be contained in the kernel $K$. Hence $K=\Sp(p-1,2)$. Thus  we have ${\bar Y}_0=\Sp(p-1,2)$. 

\q By block considerations, if $Y\neq \Sp(p,1)\oplus {\bar Y}_0$ then $Y$ contains (as a block component) a copy of $\Sp(p-1,1,1)$ or $\Sp(p-2,3)$. We leave it to the reader to check that $\dim(p,1)+\dim(p-1,2)+\dim(\lambda)$ exceed the bound given in ($\dagger$), for $\lambda=(p-1,1,1), (p-2,3)$ and hence we must have 
$$Y= \Sp(p,1)\oplus {\bar Y}_0=\Sp(p,1)\oplus \Sp(p-1,2).$$

\q Finally we consider the exact sequence 
$$0\rightarrow X\rightarrow \Sp(p,p-1,2)\rightarrow \Sp(p)\otimes U\rightarrow 0$$
of $\Sym_p\times \Sym_{p+1}$-modules. Applying  $\FF_p$  we get the exact sequence,
$$0\rightarrow Y \rightarrow \FF_p(p,p-1,2)\rightarrow \Sp(p-1,2).$$ 
For dimension reasons the map  $\FF_p(p,p-1,2)\rightarrow \Sp(p-1,2)$ can not be surjective. The  Specht module $\Sp(p-1,2)$ has exactly two composition factors with $\soc(\Sp(p-1,2))=D^{(p+1)}$ and  
$\hd(\Sp(p-1,2))=D^{(p-1,2)}$. Thus we need to rule out the possibility that there is a short exact sequence 
$$0\to \Sp(p-1,2)\oplus \Sp(p,1)\to \FF_p(p,p-1,2)\to D^{(p+1)}\to 0 \eqno{(*)}.$$

\q We first suppose that this splits. In that case the trivial module $\Sp(p+1)$  occurs twice in the socle of  of $ \FF_p(p,p-1,2)$ and we have  \\
$\dim \Hom_{\Sym_p\times\Sym_{p+1}}(k,\Sp(p,p-1,2))=2$.  However, we have 
\begin{align*}
 \Hom_{\Sym_p\times \Sym_{p+1}}(k,\Sp(p,p-1,2))&=\Hom_{\Sym_{2p+1}}(M(p+1,p),\Sp(p,p-1,2))\cr
 &=\Hom_{\GL_{2p+1}(k)}(S^{p+1}E\otimes S^pE,\nabla(p,p-1,2)).
 \end{align*}
  The module $S^{p+1}E\otimes S^pE$ has a good filtration with sections $\nabla(p+1+i,p-i)$ for $0\leq i\leq p$. It is easy to see that $\core(2p,1)=\core(p,p-1,2)=(p,1)$ and that $\core(p+i+1,p-i)\neq(p,1)$ for $i\neq p-1$. Hence we have  
 \begin{align*}
 \Hom_{\GL_{2p+1}(k)}&(S^{p+1}E\otimes S^pE,\nabla(p,p-1,2))\cr
 &=\Hom_{\GL_{2p+1}(k)}(\nabla(2p,1),\nabla(p,p-1,2)).
 \end{align*}
 But now 
 $$\Hom_{\GL_{2p+1}(k)}(\nabla(2p,1),\nabla(p,p-1,2))=\Hom_{\SL_3}(\nabla(2p-1,1),\nabla(2p+1,p-3))$$
  and by \cite{CP},  Theorem 5.1,  we get that this has dimension at most $1$.

  \q It remains to consider the possibility that $(*)$ is non-split.   Note that $\Sp(p-1,2)$ and $D^{(p+1)}$ lie in the principal block and $\Sp(p,1)$ lies in a different block so we would have $\FF_p(p,p-1,2)=\Sp(p,1)\oplus N$, for an $\Sym_{p+1}$-module $N$ such that there is a non-split extension 
  $$0\to \Sp(p-1,2)\to N\to \Sp(p+1)\to 0.$$
    It is easy to check that $\Ext^1_{\Sym_{p+1}}(\Sp(p+1),\Sp(p-2,1))=k$ so that $N$ is the unique (up to isomorphism) non-split extension.  We consider now the permutation module $M(p-1,2)$. This has a Specht series with sections (from top to bottom) $\Sp(p+1)$, $\Sp(p,1)$, $\Sp(p-1,2)$.   Since $M(p-1,2)$ is induced from the subgroup $\Sym_{p-1}\times \Sym_2$, of order prime to $p$,  it is projective.  Since the trivial $\Sym_{p+1}$-module is not projective, and by block considerations, we must have $M(p-1,2)=\Sp(p+1,1)\oplus N$, i.e., $M(p-1,2)=\FF_p(p,p-1,2)$.  Thus we would  have
  \begin{align*}
  \dim \Hom_{\Sym_p\times \Sym_{p+1}}&(k\otimes M(p-1,2),\Sp(p,p-1,2))\cr \geq &\dim \Hom_{\Sym_{p+1}}(M(p-1,2),M(p-1,2))
  \end{align*}
  and it is easy to check that the right hand side (which is independent of characteristic since $M(p-1,2)$   is a permutation module) is  $3$. By Frobenius reciprocity we would therefore have 
  $$\dim \Hom_{\Sym_{2p+1}}(M(p,p-1,2),\Sp(p,p-1,2))\geq 3.$$
  We check now that this is not the case.  We have 
  \begin{align*} \Hom_{\Sym_{2p+1}}&(M(p,p-1,2),\Sp(p,p-1,2))\cr
  =&\Hom_G(S^p E\otimes S^{p-1} E\otimes S^2 E, \nabla(p,p-1,2))
  \end{align*}
  where $G=\GL_{2p+1}(k)$ and so 
  $$ \Hom_{\Sym_{2p+1}}(M(p,p-1,2),\Sp(p,p-1,2)) = \Hom_G(S, \nabla(p,p-1,2))$$
  where $S$ is the block component of  $S^p E\otimes S^{p-1} E\otimes S^2 E$ for the block containing $\nabla(p,p-1,2)$.   Now $S$ is a summand of the injective module  $S^pE \otimes S^{p-1} E\otimes S^2 E$, see \cite{D3}, 2.1,(8),  and hence injective. Moreover, one has e.g.,   by repeated use of Pieri's formula, that $S$ has a good filtration with sections $\nabla(2p,1)$, appearing with multiplicity $2$ and $\nabla(p,p-1,2)$. Furthermore $\nabla(p,p-1,2)$ and hence the injective envelope $I(p,p-1,2)$ embeds in  $S^pE\otimes S^{p-1} E\otimes S^2 E$. Since $(p,p-1,2)$ is restricted $I(p,p-1,2)$ is a tilting module  $T(\mu)$ with highest weight $\mu=\Mull(\lambda')$, where $\lambda=(p,p-1,2)$, see \cite{D3}, 4.3,(4)(ii) and 4.3,(10),(iii).    One calculates that in fact $\mu=(2p,1)$ so that we have $S=T(2p,1)\oplus \nabla(2p,1)$ and the tilting module $T(2p,1)$ has sections $\Delta(2p,1)$, $\Delta(p,p-1,2)$. Hence we have $\Hom_G(T(2p,1),\nabla(p,p-1,2))=k$.  Moreover, $\nabla(2p,1)$ has composition factor $L(p,p-1,2)$, occurring with multiplicity $1$ and appears as the head of $\nabla(2p,1)$.   It follows that $\Hom_G(\rad\,  \nabla(2p,1), \nabla(p,p-1,2))=0$ (since the socle $L(p,p-1,2)$ of $\nabla(p,p-1,2)$ is not a composition factor of $\rad \,\nabla(2p,1)$) and $\Hom_G(\nabla(2p,1), \nabla(p,p-1,2))=k$.  Thus we have 
  $$\Hom_G(S^p E\otimes S^{p-1} E\otimes S^2 E, \nabla(p,p-1,2))=\Hom_G(S,\nabla(p,p-1,2))=k\oplus k.$$
  This rules out the possibility $(*)$ and we must have 
  $$\FF_p(p,p-1,2)=\Sp(p-1,2)\oplus\Sp(p,1).$$
  The result follows.

\end{proof}

 \begin{lemma}
 We have $F_p(p,p,2)=\dim(p,2)+\dim(p,1,1)$.
  \end{lemma}
 
 \begin{proof}  The partition $(p,p,2)$ has $2$  removable nodes $(2,p)$ and $(3,2)$ with distinct  residues $p-2$ and $p-1$.  Therefore by Corollary 1.3.6 we have 
$$F_p(p,p,2)=F_p(p,p-1,2)+F_p(p,p,1)$$
and by Lemmas  2.4 and 2.10 we get
\begin{align*}
F_p(p,p,2)&=\dim(p-1,2)+\dim(p,1)+\dim(p,1,1)\cr
&=\dim(p,2)+\dim(p,1,1)
\end{align*}
as required.
\end{proof}

\begin{lemma}   For  $p\geq a\geq b\geq 2$, $b\neq p$,  we have 
$$F_p(p,a,b)=
\begin{cases}
 \dim(a,b), & \mbox{if } b\leq a\leq p-2; \cr
\dim(p-1,b)+\dim(p+b-2,1), &\mbox{if } a=p-1 ;\cr 
\dim(p,b)+\dim(p+b-2,1,1), &\mbox{if } a=p.
\end{cases}$$
\end{lemma}

\begin{proof}
For $b=2$ the statement  is true by previous Lemmas. Hence we may assume that $b>3$. We prove the result by induction on $a+b$.

\bigskip

{\bf Case I.}    \q  $a\neq b$.

\bigskip

{\it Subcase  1.}  $a\leq p-2$.

\q Since $b\neq a$ the partition $(p,a,b)$ has three removable nodes $(1,p)$, $(2,a)$ and $(3,b)$ with the residues $p-1$, $a-2$, $b-3$, and these are distinct   since $3\leq b< a\leq p-2$. Therefore  by Corollary 1.3.6, Lemma 2.1 and induction we have
$$
F_p(p,a,b)=\dim(a-1,b)+\dim(a,b-1)=\dim(a,b)
$$
as required.\\

{\it Subcase 2.}  $a=p-1$ .

\q  Since $b<p-1$ the partition $(p,p-1,b)$ has three removable nodes $(1,p)$, $(2,p-1)$ and $(3,b)$ with distinct  residues $p-1$, $p-3$, $b-3$.  Therefore, by Corollary 1.3.6 we have 
$$F_p(p,p-1,b)=F_p(p-1,p-1,b)+ F_p(p,p-2,b)+ F_p(p,p-1,b-1).$$
By  induction and Lemma 2.1 we obtain
\begin{align*}
F_p(p,p-1,b)
&=1+\dim(p-2,b)+\dim(p-1,b-1)+\dim(p+b-3,1)\cr
&=\dim(p-2,b)+\dim(p-1,b-1)+\dim(p+b-3,1)\cr
&\ \ \ \ +\dim(p+b-2)\cr
&=\dim(p-1,b)+\dim(p+b-2,1).
\end{align*}

{\it Subcase   3.}    $a=p$ .

\q Since $b\leq p-1$,  the partition $(p,p,b)$ has two removable nodes $(2,p)$ and $(3,b)$  with distinct residues $p-2$ and $b-3$.  Hence by Corollary 1.3.6 we have 
\begin{align*}
F_p(p,p,b)&=\dim(p-1,b)+\dim(p+b-2,1)+\dim(p,b-1)\cr
&\ \ \ \ \  +\dim(p+b-3,1,1)\cr
&=\dim(p,b)+\dim(p+b-2,1,1)
\end{align*}

and  we are done.

\bigskip

{\bf Case II.}   $a=b\leq p-1.$

\bigskip

{\it Subcase  1.}   $a\leq p-2$.

\q The  partition $(p,a,a)$ has two removable nodes $(1,p)$, $(3,a)$ with distinct residues $p-1$, $a-3$. Therefore, by Corollary 1.3.6 we have 
$$F_p(p,a,a)=F_p(p-1,a,a)+F_p(p,a,a-1).$$
Applying induction and Lemma 2.1 we get 
$$F_p(p,a,a)=\dim(a,a-1)=\dim(a,a)$$
as required.

\bigskip

{\it Subcase   2.}   $a=b=p-1$. 

\q The  partition $(p,p-1,p-1)$ has two removable nodes $(1,p)$, $(3,p-1)$ with distinct  residues $p-1$, $p-4$. Therefore, by Corollary 1.3.6 we have
\begin{align*}
F_p(p,p-1,p-1)&=1+\dim(p-1,p-2)+\dim(2p-4,1)\cr
&=\dim(p-1,p-2)+\dim(2p-4,1)+\dim(2p-3)\cr
&=\dim(p-1,p-1)+\dim(2p-3,1)
\end{align*}
as required and  we are done.
\end{proof}

\begin{lemma}
 We have,
  $$F_p(p,p,p)=\dim(p,p-1)+\dim(2p-3,1,1).$$
 \end{lemma}
 
 \begin{proof}  We have $F_p(p,p,p)=F_p(p,p,p-1)$ and the result is immediate from the previous lemma.
  \end{proof}

\bs\bs\bs\bs


\section{A Counterexample for $p\geq5$.}

 We shall need the dimensions of certain simple modules.

\begin{lemma}We have 
\begin{align*}
\dim D^{(p-1,p-1,2)}&=\dim D^{(p-1,p-2,2)}\cr
&=\dim(p-1,p-2,2)-\dim(p,p-2,1)+\dim(2p-3,1,1).
\end{align*}
\end{lemma}

\begin{proof}   The partition $(p-1,p-1,2)$ has two removable nodes $R=(2,p-1)$ and $S=(3,2)$ with residues $p-3$ and $p-1$ respectively. The node $S$ is not normal since the addable node $A=(1,p)$ has residue $p-1$. Hence $R=(2,p-1)$ is the unique normal node of $(p-1,p-1,2)$. Therefore we have by Lemma 1.3.8 that 
$$\Res_{\Sym_{2p-1}}(D^{(p-1,p-1,2)})\cong D^{(p-1,p-2,2)},$$
and so $\dim D^{(p-1,p-1,2)}=\dim D^{(p-1,p-2,2)}$.

\q We show now the second equality. Note first that $\core(p-1,p-2,2)=(p-3,1,1)$. Moreover note that if $\lambda$ is a partition with $\lambda\geq (p-1,p-2,2)$ and $\core(\lambda)=(p-3,1,1)$ then $\lambda=(p-1,p-2,2)$ or $(p,p-2,1)$ or $(2p-3,1,1)$.

We have the following decomposition matrix,

$$\begin{array}{l|ccc}
\ &D^{(2p-3,1,1)}&D^{(p,p-2,1)}&D^{(p-1,p-2,2)}\\
\hline

\Sp(2p-3,1,1)&1&0&0 \\ \Sp(p,p-2,1)&1&1&0 \\ \Sp(p-1,p-2,2)&0&1&1\end{array}$$

 That the matrix above must have $1$'s in the diagonal and $0$'s  above follows directly by e.g., \cite{James},  12.1 Theorem.
  Hence $\Sp(2p-3,1,1)=D^{(2p-3,1,1)}$. The rest of the matrix can be easily obtained now using the column removal principal for decomposition numbers, see e.g. \cite{JamesII}, Theorem 6,  and applying the Schur functor. (The reader familiar with blocks of defect $1$ will easily be able to give an alternative derivation of this matrix.)
 
\q Hence we have 
\begin{align*}
\dim D^{(p-1,p-2,2)}&=\dim(p-1,p-2,2)-\dim D^{(p,p-2,1)}\cr
&=\dim(p-1,p-2,2)-\dim(p,p-2,1)+\dim(2p-3,1,1).
\end{align*}

 \end{proof}

 \begin{lemma}
 We have 
$$\dim(p,p-1,1)=F_p(p,p,p)+\dim(\rad\, (\Sp(p-1,p-1,2))).$$
 
 \end{lemma}
 
 \begin{proof} By Lemma 3.1 and Lemma 2.13 we have 
  \begin{align*}
F_p&(p,p,p)+\dim( \rad\,(\Sp(p-1,p-1,2)))\cr
 &=\dim(p,p-1)+\dim(2p-3,1,1)+\dim(p-1,p-1,2)\cr
 &\ \ \ \ \ \  -\dim D^{(p-1,p-1,2)}\cr
 &=\dim(p,p-1)+\dim(p-1,p-1,2)-\dim(p-1,p-2,2)\cr
 &\ \ \ \ \ \ \ +\dim(p,p-2,1)\cr
 &=\dim(p,p-1)+\dim(p-1,p-1,1)+\dim(p,p-2,1)\cr
 &=\dim(p,p-1,1)
 \end{align*}
  as required.
 
 \end{proof}
 
 \begin{lemma}
 There is a short exact sequence of $\Sym_{2p}$-modules,
 $$0\rightarrow \FF_p(p,p,p)\rightarrow\Sp(p,p-1,1)\rightarrow \rad\, \Sp(p-1,p-1,2)\rightarrow0.$$
 \end{lemma}
 
\begin{proof}
 Consider the principal block $\Sym_p\times \Sym_{2p}$-filtration of $\Sp(p,p,p)$ in the usual way. We have the filtration
$$\begin{array}{ccc}
\Sp(p)\otimes \Sp(p,p)\cr
\hrulefill\cr
\Sp(p-1,1)\otimes \Sp(p,p-1,1)\cr
\hrulefill\cr
\Sp(p-2,1,1)\otimes  \Sp(p-1,p-1,2)\cr
\end{array}$$
of the principal block component of $\Sp(p,p,p)$.

\q  Let  $X$ be the  $\Sym_p\times \Sym_{2p}$-module corresponding to  the two bottom sections of the above filtration. Thus we have the short exact sequence,\\
$$
0\rightarrow \Sp(p-2,1,1)\otimes \Sp(p-1,p-1,2)\rightarrow X\rightarrow\Sp(p-1,1)\otimes \Sp(p,p-1,1)\rightarrow 0.
$$
\q Let  $Y=\FF_p(X)$. Then, $Y\leq\FF_p(p,p,p)$ and by Lemma 1.4.1 we have an exact sequence

$$0\rightarrow Y\rightarrow \Sp(p,p-1,1)\rightarrow \Sp(p-1,p-1,2).$$

Note that $\hd(\Sp(p,p-1,1))\neq \hd(\Sp(p-1,p-1,2))$ and so in fact we have an exact sequence 

$$0\rightarrow Y\rightarrow \Sp(p,p-1,1)\rightarrow \rad\, \Sp(p-1,p-1,2).$$

 Hence we get 
$$
 \dim Y \geq \dim(p,p-1,1)-\dim (\rad\, \Sp(p-1,p-1,2)))=F_p(p,p,p)
$$
 
 by Lemma 3.2. Therefore $Y=\FF_p(p,p,p)$ and the sequence
 $$0\rightarrow Y\rightarrow \Sp(p,p-1,1)\rightarrow \rad\ \Sp(p-1,p-1,2)\rightarrow 0$$
is exact.

 \end{proof}

 \begin{proposition} For $p\geq 5$, the $\Sym_{2p}$-module $\FF_p(p,p,p)$ does not have a Specht filtration. 
 
 \end{proposition}
 
 \begin{proof}
 By the previous Lemma we have an exact sequence,
  \begin{align*}
 0&\rightarrow \FF_p(p,p,p)\rightarrow\Sp(p,p-1,1)\cr
 &\rightarrow \Sp(p-1,p-1,2)\rightarrow D^{(p-1,p-1,2)}\rightarrow 0.
 \end{align*}
 \q We restrict these modules to $\Sym_{2p-1}$. Note that 
 $$\Res_{\Sym_{2p-1}}(\Sp(p,p-1,1))=\Sp(p-1,p-1,1)\oplus\Sp(p,p-2,1)\oplus \Sp(p,p-1)$$
 that 
 $$\Res_{\Sym_{2p-1}}(\Sp(p-1,p-1,2))=\Sp(p-1,p-2,2)\oplus \Sp(p-1,p-1,1)$$ 
 and, by the proof of Lemma 3.1,  that 
  $$\Res_{\Sym_{2p-1}}(D^{(p-1,p-1,2)})=D^{(p-1,p-2,2)}.$$
  
  \q Therefore, as an $\Sym_{2p-1}$-module, $\FF_p(p,p,p)$ has at most three block components so we have an $\Sym_{2p-1}$-module decomposition 
  $\FF_p(p,p,p)=Y_1\oplus Y_2\oplus Y_3$. where the non-zero terms among $Y_1,Y_2,Y_3$ belong to different blocks.  Moreover, we have an exact sequence of $\Sym_{2p-1}$-modules
  \begin{align*} 0\rightarrow Y_1&\oplus Y_2\oplus Y_3\rightarrow\Sp(p-1,p-1,1)\oplus\Sp(p,p-2,1)\oplus\Sp(p,p-1)\cr
 &\rightarrow \Sp(p-1,p-2,2)\oplus\Sp(p-1,p-1,1)\rightarrow D^{(p-1,p-2,2)}\rightarrow 0.
 \end{align*}

By block considerations we have that the copies of $\Sp(p-1,p-1,1)$ must match up isomorphically and we may choose the labelling so that $Y_3= 0$. Factoring out  terms isomorphic to $\Sp(p-1,p-1,1)$ from the above sequence we get an exact sequence 
 \begin{align*}
 0\rightarrow Y_1&\oplus Y_2\rightarrow\Sp(p,p-2,1)\oplus\Sp(p,p-1)\cr
 &\rightarrow \Sp(p-1,p-2,2)\rightarrow D^{(p-1,p-2,2)}\rightarrow 0.
 \end{align*}
 Now $\Sp(p,p-2,1)$ and $\Sp(p-1,p-2,2)$ are in the same block and so we may choose the labelling so that  $Y_1=\Sp(p,p-1)$. Factoring out terms isomorphic to $\Sp(p,p-1)$ we thus obtain an exact sequence
 $$0\rightarrow Y_2\rightarrow \Sp(p,p-2,1)\rightarrow \rad\ \Sp(p-1,p-2,2)\rightarrow 0.$$

\q  If $\FF_p(p,p,p)$ has a Specht filtration  so does its restriction to $\Sym_{2p-1}$ and by projecting onto block components we obtain that each $Y_i$ has a Specht filtration.  In particular $Y_2$ has a Specht filtration and so  must contain a Specht module $\Sp(\lambda)$ for some $\lambda$ with $\deg(\lambda)=2p-1$. Then \\
$\Hom_{\Sym_{2p-1}}(\Sp(\lambda),\Sp(p,p-2,1))\neq 0$. Therefore, by Lemma 1.3.2(ii),  $\lambda\geq (p,p-2,1)$ and $\lambda\neq (p,p-2,1)$ for dimension reasons. Moreover, we have  $\core(\lambda)=\core(p,p-2,1)=(p-3,1,1)$. This implies that $\lambda=(2p-3,1,1)$. Furthermore  by Lemma 3.1 we have 
 $$\dim(2p-3,1,1)=\dim(p,p-2,1)-\dim (\rad\ \Sp(p-1,p-2,2))$$
  and so $Y_2=\Sp(2p-3,1,1)$. 
 
 \q Thus  if  $\FF_p(p,p,p)$ has a Specht filtration then this can have at most two terms. Since no partition $\lambda$ of $2p$ such that the restriction of  $\Sp(\lambda)$ to $\Sym_{2p-1}$ has  $\Sp(2p-3,1,1)$ and $\Sp(p,p-1)$ (c.f. Theorem 1.3.5). So there would exist  partitions $\lambda$ and $\mu$ such that $\Res_{\Sym_{2p-1}}(\Sp(\lambda))=\Sp(p,p-1)$ and $\Res_{\Sym_{2p-1}}(\Sp(\mu))=\Sp(2p-3,1,1)$ but, by Theorem 1.3.5 for example,  this is impossible.
 \end{proof}

 \bs\bs\bs\bs


\section{A Counterexample in characteristic 3}

\q  In this section the characteristic of $k$ is $3$. Our  counterexample in this  case is the module of $\Sym_3$ invariants in the  $\Sym_{12}$-module $\Sp(4,4,4)$.   First we calculate $F_3(4,4,4)$. We do this in several steps. The methods and techniques here are exactly the same with those in  Sections 2 and 3. Therefore, we state the first five Lemmas without giving proofs and we leave it to the reader to confirm the results by using the same arguments as  those of Section 2.  We then use these results to deduce the main point, namely that $F_3(4,4,4)=126$.

\begin{lemma} For $\lambda=(3,a)$ with $0<a\leq 3$ we have,
 $$F_3(3,a)=
 \begin{cases} 
 a, &\mbox{if } 1\leq a\leq 2;\cr
2, & \mbox{if } a=3 . 
\end{cases}$$ 
 
 \end{lemma}

\begin{lemma} Let $1\leq b\leq a\leq 3$ with  $b\neq 3$.
 Then we have
$$F_3(3,a,b)=
\begin{cases}
 1, & \mbox{if } a=b=1;\cr
b+3,&\mbox{if } a=2;\cr
5b+1,&\mbox{if } a=3.
\end{cases}$$

\end{lemma}

\begin{lemma} We have $F_3(3,3,3)=11$.

\end{lemma}

\begin{lemma} For $0\leq a\leq 4$ we have 
$$F_3(4,a)=\begin{cases}5, &  \hbox{ if } a=3; \cr
a+1,  & \hbox{ if } a\neq 3.
\end{cases}$$

\end{lemma}

\begin{lemma} We have $F_3(4,1,1)=3$.
\end{lemma}

\q With the information about the dimensions  obtained so far we are ready to calculate $F_3(4,4,4)$ and prove that $\Sp(4,4,4)$ is the desired counterexample.

\begin{lemma}
We have $F_3(4,4,4)=126$.
\end{lemma}

\begin{proof}

 We find first an upper bound for $F_3(4,4,4)$ by calculating upper bounds for $F_3(4,a,b)$ and using the left exactness of the functor $\FF_3$.  
 
 \q By Corollary 1.3.6  and the previous lemmas we have 
$$
F_3(4,2,1)\leq F_3(4,2)+F_3(4,1,1)+F_3(3,2,1)
 =10.$$

 \q By Corollary 1.3.6,  the above and the previous lemmas we have 
$$
 F_3(4,2,2)
 \leq F_3(4,2,1)+F_3(3,2,2)=15.
$$

 \q Similarly we obtain
$$
F_3(4,3,1)\leq F_3(4,3)+F_3(4,2,1)
+F_3(3,3,1)=21
$$
and 
$$
F_3(4,3,2)\leq F_3(4,3,1)+F_3(4,2,2)
 +F_3(3,3,2)=47
$$
and 
$$
F_3(4,3,3)\leq F_3(4,3,2)+F_3(3,3,3)=58
$$
and 
$$
F_3(4,4,1)\leq F_3(4,4)+F_3(4,3,1)=26
$$
and 
$$
F_3(4,4,2)\leq F_3(4,4,1)+\F_3(4,3,2)
 =73.
$$
and 
$$
F_3(4,4,3)\leq F_3(4,4,2)+ F_3(4,3,3)=131.
$$

\q We consider now the usual  $\Sym_3\times \Sym_8$ filtration of principal $\Sym_3$ block of $\Sp(4,4,3)$. This has the form
$$\begin{array}{ccc}
\Sp(3)\otimes U\cr
\hrulefill\cr
\Sp(2,1)\otimes V\cr
\hrulefill\cr
\Sp(1,1,1)\otimes W\cr
\end{array}$$
with $U\sim \Sp(4,4)|\Sp(4,3,1)$, $V\sim \Sp(4,3,1)|\Sp(4,2,2)|\Sp(3,3,2)$ and $W\sim \Sp(3,3,2)$.

\q We have  $\core(4,3,1)=(2)$, $\core(4,2,2)=(1,1)$ and $\core(3,3,2)=(3,1,1)$. Hence the Specht modules $\Sp(4,3,1),\Sp(4,2,2)$  and $\Sp(3,3,2)$ are in different blocks and so $V\cong \Sp(4,3,1)\oplus\Sp(4,2,2)\oplus \Sp(3,3,2)$. Moreover $\core(4,4)=(1,1)$ and so $U\cong \Sp(4,4)\oplus \Sp(4,3,1)$.  

\q Let  $X$ be the $\Sym_3\times \Sym_{8}$-module corresponding to   the  bottom two sections of the above filtration,. Thus we have the short exact sequence
$$0\rightarrow \Sp(1,1,1)\otimes  W\rightarrow X\rightarrow\Sp(2,1)\otimes V\rightarrow 0.$$

\q Let  $Y=\FF_3(X)$. Then $Y\leq\FF_3(4,4,3) $and  by Lemma 1.4.1 we get an exact sequence
 $$0\rightarrow Y\rightarrow V{\rightarrow} W.$$
By block considerations we get immediately that $\Sp(4,3,1)\oplus\Sp(4,2,2)\subseteq Y$ and so 
$$
\dim Y \geq \dim(4,3,1)+\dim(4,2,2)=70+56=126.
$$
Moreover using the hook formula we have that $\dim (3,3,2)=42$ and since $\dim Y \leq 131$ we conclude, by Lemma 1.3.2, that the copies of $\Sp(3,3,2)$ of $V$ and $W$ must match up isomorphically and so  $Y=\Sp(4,3,1)\oplus\Sp(4,2,2)$.

\q We consider now the short exact sequence
$$0\rightarrow X\rightarrow \Sp(4,4,3)\rightarrow \Sp(3)\otimes U\rightarrow 0.$$
Applying   $\FF_3$ we get the an exact sequence
$$0\rightarrow Y\rightarrow \FF_3(4,4,3)\rightarrow \Sp(4,4)\oplus \Sp(4,3,1)$$
By \cite{James},  p.144,  we have that $\Sp(4,4)$ has exactly two composition factors and in particular $\soc(\Sp(4,4))=D^{(6,2)}$ and $\hd(\Sp(4,4))=D^{(4,4)}$ with $\dim D^{(6,2)}=13$ and $\dim D^{(4,4)}=1$. Moreover $\Sp(4,3,1)$ has three composition factors $D^{(5,3)}, D^{(5,2,1)}$ and $\hd(\Sp(4,3,1))=D^{(4,3,1)}$ with $\dim D^{(5,3)}=28$, $\dim D^{(5,2,1)}=35$ and $\dim D^{(4,3,1)}=7$.
Since $F_3(4,4,3)\leq 131$ and $\dim Y=126$ we get immediately that $\FF_3(4,4,3)=Y$ and so $F_3(4,4,3)=126$.

\q Thus we have $F_3(4,4,4)=F_3(4,4,3)=126$ as required.
\end{proof}

\begin{proposition}
The $\Sym_9$-module $\FF_3(4,4,4)$ does not have a Specht filtration.
\end{proposition}

\begin{proof}

By the previous Lemma we have that  $F_3(4,4,4)=126$. We consider now the $\Sym_3\times \Sym_9$-filtration 
$$\begin{array}{ccc}
\Sp(3)\otimes \Sp(4,4,1)\cr
\hrulefill\cr
\Sp(2,1)\otimes \Sp(4,3,2) \cr
\hrulefill\cr
\Sp(1,1,1)\otimes \Sp(3,3,3)\cr
\end{array}$$
of $\Sp(4,4,4)$.   Let  $X$ be the $\Sym_3\times \Sym_{9}$-module corresponding to the   bottom two  sections of the above filtration. Thus  we have a  short exact sequence
$$0\rightarrow \Sp(1,1,1)\otimes \Sp(3,3,3)\rightarrow X\rightarrow\Sp(2,1)\otimes \Sp(4,3,2)\rightarrow 0.$$
Let $Y=\FF_3(X)$. Then $Y\leq\FF_3(4,4,4)$ and by Lemma 1.4.1 we have  an exact sequence
 $$0\rightarrow Y\rightarrow \Sp(4,3,2)  {\rightarrow} \Sp(3,3,3).$$
Hence 
$$\dim Y \geq \dim(4,3,2)-\dim(3,3,3)=168-42=126$$
and therefore  $Y=\FF_3(4,4,4)$ and we have a short exact sequence 
$$0\rightarrow \FF_3(4,4,4)\rightarrow \Sp(4,3,2){\rightarrow} \Sp(3,3,3)\rightarrow 0.$$
By \cite{James},  p.145,  the composition factors of $\Sp(4,3,2)$ are:
$$
D^{(8,1)}, D^{(7,1,1)}, D^{(6,3)}, D^{(6,2,1)}, D^{(5,4)}, D^{(5,2,2)}, D^{(4,,4,1)}  \hbox{ and } D^{(4,3,2)}
$$
and the composition factors of $\Sp(3,3,3)$ are
$D^{(7,1,1)}$  and $D^{(4,3,2)}.$  Therefore the composition factors of $\FF_3(4,4,4)$ are 
$$
D^{(8,1)}, D^{(6,3)}, D^{(6,2,1)}, D^{(5,4)}, D^{(5,2,2)} \hbox{ and }  D^{(4,4,1)}
$$

\q Using James's  tables, \cite{James},  p.145, is easy to deduce that the only way for  $\FF_3(4,4,4)$ to have a Specht filtration is with sections
$\Sp(6,3)$, $\Sp(5,1^4)$  and $\Sp(2,1^7)$.  Therefore one of these Specht modules would embed in  $\FF_3(4,4,4)$ and so inside $\Sp(4,3,2)$. 

\q We show that this does not happen.   Assume for a contraction that it does. First note that $D^{(6,3)}$ embeds in $\Sp(4,3,2)$ and is not a composition  factor of $\Sp(3,3,3)$ and so $D^{(6,3)}$ embeds in $\FF_3(4,4,4)$.  Hence $D^{(6,3)}$ embeds in a section, namely in $\Sp(6,3)$, $\Sp(5,1)$ or $\Sp(2,1^7)$.  But $D^{(6,3)}$ appears only once as a composition factor of $\FF_3(4,4,4)$ and is a composition factor of $\Sp(6,3)$ so does not appear in the socle of $\Sp(5,1^4)$ or $\Sp(2,1^7)$.  The only possibility now is that $D^{(6,3)}$ is in the socle of $\Sp(6,3)$. But, from James's tables, $\Sp(6,3)$ is not irreducible: it has composition factors $D^{(6,3)}$ and $D^{(8,1)}$.   Moreover $\Sp(6,3)$ has simple head $D^{(6,3)}$ and so $D^{(6,3)}$  does not appear in the socle. This completes the proof.

\end{proof}

\bs\bs\bs\bs


\section{ A Counterexample in characteristic $2$.}

\q Our  counterexample in the case of $p=2$ is the module of $\Sym_2$ invariants in   $\Sp(4,4)$. 
 We first work towards the  calculation   of $F_2(4,4)$.

\begin{lemma}

We have  $\F_2(2)=F_2(1,1)=1$.

\end{lemma}

\begin{proof}

Clear.

\end{proof}

\begin{lemma}

We have  $F_2(2,1)=F_2(2,2)=1$.
\end{lemma}

\begin{proof} Since  $F_2(2,2)=F_2(2,1)$ it is enough to calculate $F_2(2,1)$.  By Lemma 1.3.5, the $\Sym_2$-module $\Sp(1,1)$ is inside $\Sp(2,1)$ so that $F_2(2,1)\neq 0$. By Corollary 1.3.6 we have $F_2(2,1)\leq F_2(2)+F_2(1,1)$ so that $F_2(2,1)\leq 2$. Moreover $\Sp(2,1)$ embeds in the permutation  module $M(2,1)$, which has a one dimension space of $\Sym_2$-invariants, Hence $F_2(2,1)\leq 1$ and so $F_2(2,1)=1$.

\end{proof}

\begin{lemma}
We have $F_2(3,1)=2$ and $F_2(3,2)=F_2(3,3)=3$.
\end{lemma}

\begin{proof}

 The partition $(3,1)$ has two removable nodes $(1,3)$ and $(2,1)$ with residues $0$ and $1$. Therefore by Corollary 1.3.6 we have $F_2(3,1)=F_2(3)+F_2(2,1)=2$.

\q The partition $(3,2)$ has two removable nodes $(1,3)$ and $(2,2)$. Hence by Corollary 1.3.6  we have  $F_2(3,2)\leq F_2(3,1)+ F_2(2,2)=3$. 

\q We consider now the $\Sym_2$-invariants of the rational Specht module  $\Sp_\que(3,2)$. We have that the Specht module $\Sp_\que (3,2)$ decomposes as 
$$\Sp_\que(3,2)\cong \Sp_\que(2)\otimes (\Sp_\que(3)\oplus\Sp_\que(2,1))\oplus \Sp_\que(1,1)\otimes \Sp_\que(2,1)$$
and so $\Sp_\que(3,2)^{\Sym_2}\cong \Sp_\que(3)\oplus \Sp_\que(2,1)$, and so is three dimensional.  By Remark 1.3.7 we have 
$\F_2(3,2)\geq 3$ and so $F_2(3,2)=3$.

\q Hence we also have $F_3(3,3)=F_3(3,2)=3$.

\end{proof}

\begin{lemma}
 We have $F_2(4,1)=3$, $F_2(4,2)=6$, 
$\F_2(4,3)=9$ and $F_2(4,4)=9$.

\end{lemma}

\begin{proof}
The partition $(4,1)$ has two removable nodes $(1,4)$ and $(2,1)$.  Hence by Corollary 1.3.6  we have 
$$F_2(4,1)\leq F_2(4)+ F_2(3,1)=3.$$
\q We consider now the $\Sym_2$-invariants of $\Sp_\que(4,1)$. We have that the Specht module $\Sp_\que(4,1)$ decomposes as a $\Sym_2\times \Sym_3$-module as
$$\Sp_\que(4,1)\cong \Sp_\que(2)\otimes (\Sp_\que(3)\oplus\Sp_\que(2,1))\oplus \Sp_\que(1,1)\otimes \Sp_\que(3)$$
and so $\Sp_\que(4,1)^{\Sym_2}\cong \Sp_\que(3)\oplus \Sp_\que(2,1)$. Hence we have 
$$F_2(4,1)\geq \dim \Sp_\que(4,1)^{\Sym_2}=3.$$
and hence, by Remark 1.3.7, we have   $F_2(4,1)=3$.

\q The partition $(4,2)$ has two removable nodes $(1,4)$ and $(2,2)$ with residues $1$ and $0$ respectively. Therefore we have 
 $$F_2(4,2)=F_2(4,1)+F_2(3,2)=6.$$

\q By Corollary 1.3.6 we have 
$$
F_2(4,3)\leq F_2(4,2)+F_2(3,3)=9.
$$ 

\q We consider now the $\Sym_2$-invariants of  the rational Specht module $\Sp_\que(4,3)$. The  Specht module $\Sp_\que(4,3)$ decomposes as a $\Sym_2\times \Sym_5$-module ias 
$$
\Sp_\que(4,3)\cong \Sp_\que(2)\otimes (\Sp_\que(4,1)\oplus\Sp_\que(3,2))\oplus \Sp_\que(1,1)\otimes \Sp_\que(3,2)
$$
and so $\Sp_\que(4,3)^{\Sym_2}\cong \Sp_\que(4,1)\oplus \Sp_\que(3,2)$. Hence we have that 
$F_2(\Sp(4,3))\geq \dim \Sp_\que(4,3)^{\Sym_2}=9$. Therefore $F_2(\Sp(4,3))=9$.

\q Thus we also have $F_2(4,4)=F_2(4,3)=9$.
 
 \end{proof}

 \begin{proposition}
 
 The $\Sym_6$-module $\FF_2(4,4)$ does not have a Specht filtration.

\end{proposition} 

\begin{proof}
 We consider now the $\Sym_2\times \Sym_6$ filtration 

$$\begin{array}{ccc}
\Sp(2)\otimes \Sp(4,2)\cr
\hrulefill\cr
\Sp(1,1)\otimes \Sp(3,3) \cr

\end{array}$$
 of $\Sp(4,4)$.

Therefore we have a short exact sequence 

$$0\rightarrow \Sp(1,1)\otimes \Sp(3,3)\rightarrow \Sp(4,4)\rightarrow \Sp(2)\otimes \Sp(4,2)\rightarrow 0$$

Applying $\FF_2$ we then get an  exact sequence

$$0\rightarrow \Sp(3,3) \rightarrow \FF_2(4,4)\overset{f}{\rightarrow} \Sp(4,2) \eqno{(\dagger)}.$$

Using James's  tables \cite{James}, p.137,  we have that the composition factors of $\Sp(4,2)$ are
$
D^{(6)}, D^{(5,1)} \hbox{ and } D^{(4,2)}
$
and the composition factors of $\Sp(3,3)$ are
$D^{(6)}$ and $D^{(4,2)}$

\q  By the previous Lemma we have that $F_2(4,4)=9$ and so we get by  $(\dagger)$  that the composition factors of $\FF_2(4,4)$ are  $D^{(6)}, D^{(5,1)}$ and $D^{(4,2)}$ .  Using James's  tables once more, \cite{James},  p.137,  is easy to deduce that if $\FF_2(4,4)$ has a Specht filtration then $ \FF_2(4,4)=\Sp(4,2)$ or  $ \FF_2(4,4)=\Sp(2,2,1,1)$.  We show that both cases are impossible. 

\bs
 
 {\it Case} 1. Assume first that $ \FF_2(4,4)=\Sp(4,2)$. By \cite{James}, Corollary 13.17, we have 
$\End_{\Sym_6}(\Sp(4,2))=k$ and so the map $f$ in   $(\dagger)$  must be an isomorphism or the zero map. But  by the exactness of the sequence above $f$ cannot be an isomorphism. Moreover if $f=0$ then  $\Sp(4,2)\cong\Sp(3,3)$ which is again impossible since $\dim(3,3)=5$ and $\dim(4,2)=9$.

\bs

{\it Case 2.} Assume now that $ \FF_2(4,4)=\Sp(2,2,1,1)$. We have that composition factors of $\Sp(3,3)$ are $D^{(6)}$ and $D^{(4,2)}$. In particular the trivial $\Sym_6$ module $D^{(6)}$ embeds in $\Sp(3,3)$, for otherwise  we would have that $\hd(\Sp(3,3))=D^{(6)}$ and so the trivial $\Sym_5$-module $D^{(5)}$ would  appear in the head  of the Specht module $\Sp(3,2)$. But this is impossible since the partition $(3,2)$ is $2$-regular and so the Specht module $\Sp(3,2)$ has simple head $D^{(3,2)}$. Since $D^{(6)}$ embeds in $\Sp(3,3)$ and $ \FF_2(4,4)=\Sp(2,2,1,1)$ we  have an embedding of $D^{(6)}$ inside $\Sp(2,2,1,1)$. Hence $\Hom_{\Sym_6}(k,\Sp(2,2,1,1))\neq0$. Using now \cite{James},  Theorem 8.15,  we get
$$
\Hom_{\Sym_6}(k,\Sp(2,2,1,1))=\Hom_{\Sym_6}(\Sp(2,2,1,1)^*,k)=\Hom_{\Sym_6}(\Sp(4,2),k)
$$
and so $\Hom_{\Sym_6}(\Sp(4,2),k)\neq0$. But this is impossible since the Specht module $\Sp(4,2)$ has  simple head  $D^{(4,2)}$  by \cite{James}, 12.2 Corollary.  The proof is complete.

\end{proof}

\section*{Acknowledgement}

The second author gratefully acknowledges the financial  support of EPSRC Grant EP/L005328/1.

\bs\bs\bs\bs


\end{document}